\documentclass[
final
]{dmtcs-episciences}

\usepackage[utf8]{inputenc}
\usepackage{subfigure}

\usepackage{amsthm,amsmath,amssymb,mathabx}

\usepackage{tikz,multirow,ytableau}

\usepackage[round]{natbib}

\theoremstyle{plain}
\newtheorem{theorem}{Theorem}
\newtheorem{lemma}{Lemma}
\newtheorem{corollary}{Corollary}
\newtheorem{prop}{Proposition}

\author{Monica Anderson\affiliationmark{1}
  \and Marika Diepenbroek \affiliationmark{2}
	  \and Lara Pudwell\affiliationmark{3} 
  \and Alex Stoll\affiliationmark{4}}
	\title[Pattern Avoidance in Reverse Double Lists]{Pattern Avoidance in Reverse Double Lists \footnote{Research was supported by the National Science Foundation (NSF DMS-1262852)}}
\affiliation{
 Department of Mathematics, Minnesota State University Moorhead\\
Department of Mathematics, University of North Dakota\\
  Department of Mathematics and Statistics, Valparaiso University\\
  Department of Mathematical Sciences, Clemson University}
\received{2017-4-28}
\revised{2018-2-17, 2018-6-19}
\accepted{2018-10-11}
\begin{document}
\publicationdetails{19}{2018}{2}{14}{3289}
\maketitle

\begin{abstract}
In this paper, we consider pattern avoidance in a subset of words on $\{1,1,2,2,\dots,n,n\}$ called reverse double lists.  In particular a reverse double list is a word formed by concatenating a permutation with its reversal.  We enumerate reverse double lists avoiding any permutation pattern of length at most 4 and completely determine the corresponding Wilf classes.  For permutation patterns $\rho$ of length 5 or more, we characterize when the number of $\rho$-avoiding reverse double lists on $n$ letters has polynomial growth.  We also determine the number of $1\cdots k$-avoiders of maximum length for any positive integer $k$.
\end{abstract}

\section{Introduction}\label{S:Intro}

Let $\mathcal{S}_n$ be the set of all permutations on $[n]=\{1,2,\dots, n\}$.  Given $\pi \in \mathcal{S}_n$ and $\rho \in \mathcal{S}_k$ we say that $\pi$ \emph{contains} $\rho$ as a pattern if there exists $1 \leq i_1 < i_2 < \cdots < i_k \leq n$ such that $\pi_{i_a} \leq \pi_{i_b}$ if and only if $\rho_a \leq \rho_b$.  In this case we say that $\pi_{i_1}\cdots \pi_{i_k}$ is \emph{order-isomorphic} to $\rho$, and that $\pi_{i_1} \cdots \pi_{i_k}$ is an \emph{occurrence} of $\rho$ in $\pi$.  If $\pi$ does not contain $\rho$, then we say that $\pi$ \emph{avoids} $\rho$.  Of particular interest are the sets $\mathcal{S}_n(\rho)=\{\pi \in \mathcal{S}_n \mid \pi \text{ avoids } \rho\}$.  Let $\mathrm{s}_n(\rho)=\left|\mathcal{S}_n(\rho)\right|$.  It is well known that $\mathrm{s}_n(\rho)=\frac{\binom{2n}{n}}{n+1}$ for $\rho \in \mathcal{S}_3$; see \cite{K68}.  For $\rho \in \mathcal{S}_4$, 3 different sequences are possible for $\{\mathrm{s}_n(\rho)\}_{n \geq 1}$.  Two of these sequences are well-understood, but an exact formula for $\mathrm{s}_n(1324)$ remains open; see \cite{CG14}.

Pattern avoidance has been studied for a number of combinatorial objects other than permutations.  The definition above extends naturally for patterns in words (i.e. permutations of multisets) and there have been several algorithmic approaches to determining the number of words avoiding various patterns; see \cite{BM05, B98, JM09, P10}. 

In another direction, a permutation may be viewed as a bijection on $[n]$.  When we graph the points $(i,\pi_i)$ in the Cartesian plane, all points lie in the square $[0,n+1] \times [0,n+1]$, and thus we may apply various symmetries of the square to obtain involutions on the set $\mathcal{S}_n$.  For $\pi \in \mathcal{S}_n$, let $\pi^r = \pi_n \cdots \pi_1$ and let $\pi^c = (n+1-\pi_1)\cdots (n+1-\pi_n)$, the reverse and complement of $\pi$ respectively.  For example, the graphs of $\pi=1342$, $\pi^r=2431$, and $\pi^c = 4213$ are shown in Figure \ref{permgraph}.  

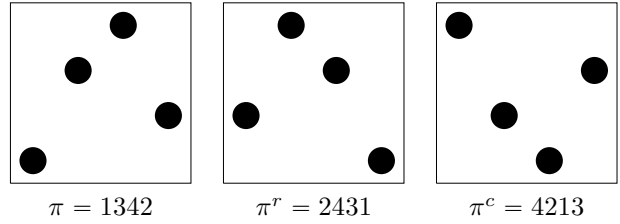
\begin{figure}[!hbt]
\begin{center}
\begin{tabular}{ccc}
\scalebox{0.6}{\begin{tikzpicture}
\draw (0,0)--(4,0)--(4,4)--(0,4)--(0,0);
\fill[black] (0.5,0.5) circle (0.3cm);
\fill[black] (1.5,2.5) circle (0.3cm);
\fill[black] (2.5,3.5) circle (0.3cm);
\fill[black] (3.5,1.5) circle (0.3cm);
\end{tikzpicture}}&
\scalebox{0.6}{\begin{tikzpicture}
\draw (0,0)--(4,0)--(4,4)--(0,4)--(0,0);
\fill[black] (0.5,1.5) circle (0.3cm);
\fill[black] (1.5,3.5) circle (0.3cm);
\fill[black] (2.5,2.5) circle (0.3cm);
\fill[black] (3.5,0.5) circle (0.3cm);
\end{tikzpicture}}&
\scalebox{0.6}{\begin{tikzpicture}
\draw (0,0)--(4,0)--(4,4)--(0,4)--(0,0);
\fill[black] (0.5,3.5) circle (0.3cm);
\fill[black] (1.5,1.5) circle (0.3cm);
\fill[black] (2.5,0.5) circle (0.3cm);
\fill[black] (3.5,2.5) circle (0.3cm);
\end{tikzpicture}}\\
$\pi=1342$& $\pi^r=2431$ &$\pi^c=4213$\\
\end{tabular}
\end{center}
\caption{The graphs of $\pi=1342$, $\pi^r=2431$, and $\pi^c=4213$}
\label{permgraph}
\end{figure}

Pattern-avoidance in centrosymmetric permutations, i.e. permutations $\pi$ such that $\pi^{rc}=\pi$ has been studied by \cite{E07, E10}, by \cite{LO10}, and by \cite{BBS10}.  \cite{F11} generalized this idea to pattern avoidance in centrosymmetric words.  More recently,  \cite{VERUM2014} defined the set of \emph{double lists} on $n$ letters to be $$\mathcal{D}_n = \{ \pi\pi \mid \pi \in \mathcal{S}_n\}.$$  In other words, a double list is a permutation of $[n]$ concatenated with itself.  We see immediately that $\left|\mathcal{D}_n\right|=n!$.  Cratty et. al. completely characterized the members of $\mathcal{D}_n$ that avoid a given permutation pattern of length at most 4.  In all of these cases, knowing the first half of a permutation or word determines the second half.  

In this paper we consider a type of word that exhibits a different symmetry.  In particular, let $$\mathcal{R}_n = \{ \pi\pi^r \mid \pi \in \mathcal{S}_n\}.$$  For example, $\mathcal{R}_3 = \left\{123321, 132231, 213312, 231132, 312213, 321123\right\}$.  We call $\mathcal{R}_n$ the set of \emph{reverse double lists} on $n$ letters.  Consider $$\mathcal{R}_n(\rho)= \{ \sigma \in \mathcal{R}_n \mid \sigma \text{ avoids } \rho\},$$ and let $\mathrm{r}_n(\rho)=\left|\mathcal{R}_n(\rho)\right|$.  We obtain a number of interesting enumeration sequences for $\{\mathrm{r}_n(\rho)\}_{n \geq 1}$ with connections to other combinatorial objects.  In Section \ref{S:monotone} we consider $\mathrm{r}_n(12\cdots k)$ for any positive integer $k$.  We give an analogue of the Erd\H{o}s--Szekeres Theorem for reverse double lists; in other words, we show that $\mathrm{r}_n(12\cdots k)=0$ for sufficiently large $n$.  We also enumerate the number of $12\cdots k$-avoiders of maximum length. In Section \ref{S:small}  we completely determine $\mathrm{r}_n(\rho)$ for $\rho \in \mathcal{S}_3 \cup \mathcal{S}_4$.  In Section \ref{S:five}, we give data that describes all Wilf classes for avoiding a pattern $\rho \in \mathcal{S}_5$; we also classify the enumeration generating functions for many of these Wilf classes. More generally, we characterize when $\mathrm{r}_n(\rho)$ has polynomial growth for a pattern $\rho$ of arbitrary length.

\section{Avoiding monotone patterns}\label{S:monotone}

In this section, we show that $\mathrm{r}_n(12\cdots k)=0$ for sufficiently large $n$.  Theorem \ref{T:erdos} gives a sharp bound on when $\mathrm{r}_n(12\cdots k)=0$, while Theorem \ref{T:syt} enumerates the maximal length avoiders of $12\cdots k$.

\begin{theorem}
$\mathrm{r}_n(12\cdots k)=0$ for $n \geq \binom{k}{2}+1$.
\label{T:erdos}
\end{theorem}

\begin{proof}
Consider $\sigma =\pi\pi^r \in \mathcal{R}_n$.  Following Seidenberg's proof of the Erd\H{o}s--Szekeres Theorem in \cite{S59}, for $1 \leq i \leq n$, let $a_i$ be the length of the longest increasing subsequence of $\pi$ ending in $\pi_i$ and let $b_i$ be the length of the longest decreasing subsequence of $\pi$ ending in $\pi_i$.  By definition, $1 \leq a_i, b_i \leq n$.  Further, if $i \neq j$, then $(a_i, b_i) \neq (a_j, b_j)$, since if $\pi_i<\pi_j$ then $a_i < a_j$ and if $\pi_i > \pi_j$ then $b_i < b_j$.  Finally, for all $i$, the increasing subsequence of length $a_i$ in $\pi$ ending at $\pi_i$ followed by the reversal of the decreasing subsequence of length $b_i$ in $\pi$ ending at $\pi_i$, minus the digit $\pi_i$ in $\pi^r$ forms an increasing subsequence of length $a_i+b_i-1$ in $\sigma$.  If $\sigma \in \mathcal{R}_n(12\cdots k)$, it must be the case that $a_i+b_i-1 < k$ for all $i$.  There are $\binom{k}{2}$ distinct pairs of positive integers where $a_i+b_i-1 < k$, so we have that $\mathrm{r}_n(12\cdots k) = 0$ for $n \geq \binom{k}{2}+1$.
\end{proof}

In fact, this bound is sharp. Let $J_\ell$ be the decreasing permutation of length $\ell$.  Also, let $\alpha \oplus \beta$ denote the direct sum of permutations $\alpha=\alpha_1 \cdots \alpha_n$ and $\beta=\beta_1\cdots \beta_m$, i.e. $$(\alpha \oplus \beta)_i = \begin{cases}
\alpha_i & 1 \leq i \leq n\\
n+\beta_{i-n} & n+1 \leq i \leq n+m
\end{cases}.$$ Then $\pi=J_{k-1} \oplus J_{k-2} \oplus \cdots \oplus J_2 \oplus J_1$ is a permutation of length $\binom{k}{2}$ such that $\pi\pi^r \in \mathcal{R}_n(12\cdots k)$.

Now that we have a sharp bound on when $\mathrm{r}_n(12\cdots k)=0$, a natural question is: how many maximal $12\cdots k$-avoiders are there?  This question is most easily answered using the Robinson--Schensted correspondence between permutations and pairs of standard Young tableaux of the same shape.  Recall that the Robinson--Schensted correspondence maps $a \mapsto \left(\ytableausetup{centertableaux}
\begin{ytableau}
a
\end{ytableau},\ytableausetup{centertableaux}
\begin{ytableau}
1\end{ytableau}\right)$.
Now, given the pair of tableau $(P,Q)$ for $\pi_1\cdots \pi_{n-1}$, we insert $\pi_n$ in the following way: 
\begin{enumerate}
\item[(a)] If $\pi_n$ is larger than all entries of row 1 of $P$, then append $\pi_n$ to the end of row 1 of $P$ and append $n$ to the end of row 1 of $Q$.  
\item[(b)] Otherwise, find the first entry $i$ of row 1 of $P$ that is larger than $\pi_n$, replace this entry with $\pi_n$ and now repeat steps (a) and (b) by trying to insert $i$ into row 2, bumping if necessary.  When we finally have an entry that is added to the end of a row, insert a box in the corresponding place in $Q$ with entry $n$.  
\end{enumerate}
In general, we write $(P(\pi),Q(\pi))$ for the pair of tableau corresponding to $\pi$.  For example, the steps of this bumping algorithm for the permutation 452316 are shown in Figure \ref{F:rsk}.

\begin{figure}[hbt]
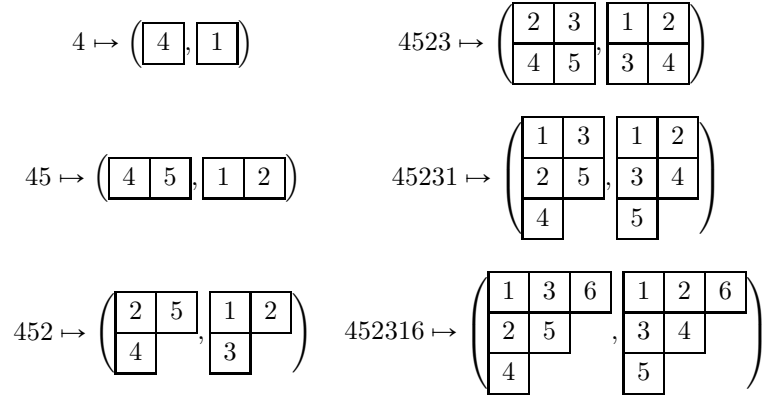

\begin{center}
\begin{tabular}{cc}

$4 \mapsto \left(\ytableausetup{centertableaux}
\begin{ytableau}
4
\end{ytableau},\ytableausetup{centertableaux}
\begin{ytableau}
1\end{ytableau}\right)$

&
$4523 \mapsto \left(\ytableausetup{centertableaux}
\begin{ytableau}
2&3\\
4&5
\end{ytableau},\ytableausetup{centertableaux}
\begin{ytableau}
1&2\\
3&4
\end{ytableau}\right)$\\

\phantom{X}&\\

$45 \mapsto \left(\ytableausetup{centertableaux}
\begin{ytableau}
4&5
\end{ytableau},\ytableausetup{centertableaux}
\begin{ytableau}
1&2
\end{ytableau}\right)$

&

$45231 \mapsto \left(\ytableausetup{centertableaux}
\begin{ytableau}
1&3\\
2&5\\
4
\end{ytableau},\ytableausetup{centertableaux}
\begin{ytableau}
1&2\\
3&4\\
5
\end{ytableau}\right)$\\

\phantom{X}&\\

$452 \mapsto \left(\ytableausetup{centertableaux}
\begin{ytableau}
2&5\\
4
\end{ytableau},\ytableausetup{centertableaux}
\begin{ytableau}
1&2\\
3
\end{ytableau}\right)$

&

$452316 \mapsto \left(\ytableausetup{centertableaux}
\begin{ytableau}
1&3&6\\
2&5\\
4
\end{ytableau},\ytableausetup{centertableaux}
\begin{ytableau}
1&2&6\\
3&4\\
5
\end{ytableau}\right)$
\end{tabular}
\end{center}
\caption{The Robinson--Schensted correspondence applied to the permutation 452316}
\label{F:rsk}
\end{figure}

In this correspondence, the number of rows of $P(\pi)$ gives the length of the longest decreasing subsequence of $\pi$, and the number of columns of $P(\pi)$ gives the length of the longest increasing subsequence of $\pi$.  If $\sigma=\pi\pi^r \in \mathcal{R}_{\binom{k}{2}}(12\cdots k)$, we expect both the longest increasing subsequence and the longest decreasing subsequence of $\pi$ to have length less than $k$.  But more can be said.  Consider the following result of Greene:

\begin{theorem}[\cite{G74}, Theorem 3.1]
Let $\pi$ be a permutation, let $P(\pi)$ have $k$ rows and let $\lambda_i$ denote the length of the $i$th row of $P(\pi)$.  Then for all $1 \leq i \leq k$, the maximum size of the union of $i$ increasing subsequences in $\pi$ is equal to $\lambda_1+\lambda_2+\cdots + \lambda_i$.
\label{T:lis}
\end{theorem}

We are ready to state a characterization of $P(\pi)$ if $\pi\pi^r \in \mathcal{R}_{\binom{k}{2}}(12\cdots k)$.

\begin{theorem}
$\mathrm{r}_{\binom{k}{2}}(12\cdots k)$ is equal to the number of pairs of standard Young tableaux $(P,Q)$ of shape $(k-1,k-2,\dots,1)$ where $P$ has increasing diagonals (i.e. the entry in row $r$ column $c$ is less than the entry in row $r-1$, column $c+1$ for $2 \leq r \leq k-1$ and $1 \leq c \leq k-2$).
\label{T:syt}
\end{theorem}

First, we show that the $(k-1,k-2,\dots,1)$ shape of $P(\pi)$ is necessary for $\pi\pi^r \in \mathcal{R}_{\binom{k}{2}}(12\cdots k)$.

\begin{lemma}\label{L:staircase}
Suppose $\pi \in \mathcal{S}_{\binom{k}{2}}$. If $P(\pi)$ does not have the shape $(k-1,k-2,\dots,1)$, then $\pi\pi^r$ contains an increasing subsequence of length $k$.
\end{lemma}

\begin{proof}
We know that $P(\pi)$ has at most $k-1$ rows; otherwise, $\pi$ contains a decreasing subsequence of length $k$, and $\pi^r$ contains an increasing subsequence of length $k$.  Now, if $\lambda_j$ is the length of row $j$ of $P(\pi)$, let $r$ be the first row such that $\lambda_j\geq k-j+1$.  We know from Theorem \ref{T:lis} that the maximum size of the union of $r$ increasing subsequences in $\pi$ is equal to $\lambda_1+\cdots + \lambda_r$.  In fact, we can find disjoint increasing subsequences of lengths $\lambda_1, \lambda_2, \dots, \lambda_r$ in $\pi$.  This means that there are at least $r$ distinct elements of $\pi$ that are the last element in an increasing subsequence of length $k-r+1$.  However, from the proof of Theorem \ref{T:erdos}, we know that there are at most $r-1$ such elements if $\pi \pi^r$ avoids $12 \cdots k$.  Therefore, $P(\pi)$ and $Q(\pi)$ have shape  $(k-1,k-2,\dots,1)$  whenever $\pi\pi^r \in \mathcal{R}_{\binom{k}{2}}(12\cdots k)$.
\end{proof}

Notice while $P(\pi)$ having the shape $(k-1,k-2,\dots,1)$ is necessary for $\pi\pi^r \in \mathcal{R}_{\binom{k}{2}}(12\cdots k)$, it is not sufficient.  For example, when $\pi=246513$, $P(\pi)$ has the shape $(3,2,1)$, but $\pi\pi^r$ contains a 1234 pattern, realized by the digits 1 and 3 in $\pi$ together with the digits 5 and 6 in $\pi^r$.  However, by Lemma \ref{L:staircase}, we may restrict our attention to $$\mathcal{S}^*_{\binom{k}{2}} = \left\{ \pi \in \mathcal{S}_{\binom{k}{2}} \middle| P(\pi) \text{ has shape } (k-1,k-2,\dots, 1)\right\}.$$  To prove Theorem \ref{T:syt}, we must show that for $\pi \in S^*_{\binom{k}{2}}$, $\pi \pi^r$ avoids $12\cdots k$ if and only if $P(\pi)$ has increasing diagonals.

Before we finish the proof of Theorem \ref{T:syt}, we recall a construction of \cite{V77} that relates the entries of $P(\pi)$ to the graph of $\pi$.  Consider the points $(i,\pi_i)$ for $1 \leq i \leq n$.  The \emph{shadow} of point $(x_i, y_i)$, denoted $S((x_i,y_i))$, is $S((x_i,y_i)):=\left\{(u,v) \in \mathbb{R}^2 \middle| u \geq x_i \text{ and } v \geq y_i\right\}$.  In other words, the shadow is the collection of all points above and to the right of the original point.  Consider the points of the graph of $\pi$ that are not in the shadow of any other point.  The boundary of the union of their shadows is the first shadow line $L_1$ of $\pi$.  To form the second shadow line (and subsequent shadow lines), remove the points on the first shadow line from the graph of $\pi$, and repeat.  An example, showing the shadow lines of $\pi=452316$, is given in Figure \ref{F:shadow}.

\begin{figure}

\begin{center}
\begin{tikzpicture}
\draw[step=1cm,gray,very thin] (0,0) grid (7,7);
\foreach \x in {1,2,3,4,5,6}
    \draw (\x cm,1pt) -- (\x cm,-1pt) node[anchor=north] {$\x$};
\foreach \y in {1,2,3,4,5,6}
    \draw (1pt,\y cm) -- (-1pt,\y cm) node[anchor=east] {$\y$};
    \draw[line width=0.1cm] (1,7) -- (1,4)--(3,4)--(3,2)--(5,2)--(5,1)--(7,1);
    \draw[line width=0.1cm] (2,7)--(2,5)--(4,5)--(4,3)--(7,3);
    \draw[line width=0.1cm] (6,7)--(6,6)--(7,6);
					\fill[black] (1,4) circle (.2cm);
					\fill[black] (2,5) circle (.2cm);
					\fill[black] (3,2) circle (.2cm);
					\fill[black] (4,3) circle (.2cm);
					\fill[black] (5,1) circle (.2cm);
					\fill[black] (6,6) circle (.2cm);
\end{tikzpicture}
\end{center}
\caption{The shadow lines of $\pi=452316$}
\label{F:shadow}
\end{figure}
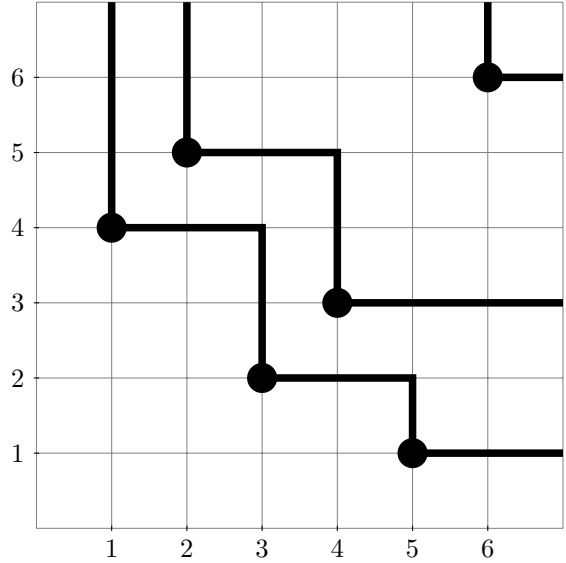

Viennot showed that the $y$-coordinates of the rightmost point on each shadow line are the entries in the first row of $P(\pi)$.  Indeed, using the Robinson--Schensted correspondence, we saw in Figure \ref{F:rsk} that $$P(452316) = \ytableausetup{centertableaux}
\begin{ytableau}
1&3&6\\
2&5\\
4
\end{ytableau}$$
and the $y$-coordinates of the rightmost points on the shadow lines in Figure \ref{F:shadow} are 1, 3, and 6.  The second row of $P(\pi)$ can be found in a similar way: mark the corners of the shadow lines where there is no point of the original permutation, as shown by the squares in the left side of Figure \ref{F:shadow2}.  Then, using these corners as the new permutation graph, draw shadow lines again, as shown in the right side of Figure \ref{F:shadow2}.  The $y$-coordinates of rightmost points on each of the new shadow lines are the second row of $P(\pi)$; in this case the entries are 2 and 5. We can iterate this shadow line process to obtain all rows of $P(\pi)$. Shadow lines are the main tool in our proof of Theorem \ref{T:syt} that follows.  We proceed with a series of lemmas.

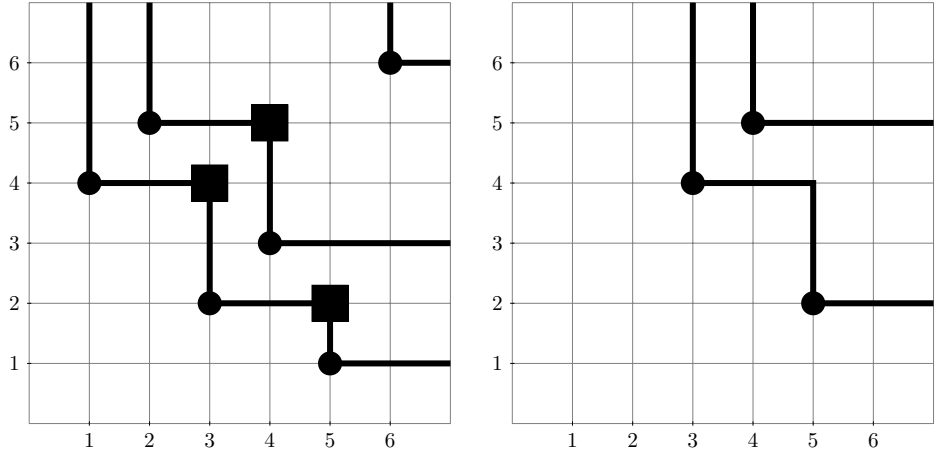
\begin{figure}

\begin{center}
\begin{tabular}{cc}
\scalebox{0.8}{
\begin{tikzpicture}
\draw[step=1cm,gray,very thin] (0,0) grid (7,7);
\foreach \x in {1,2,3,4,5,6}
    \draw (\x cm,1pt) -- (\x cm,-1pt) node[anchor=north] {$\x$};
\foreach \y in {1,2,3,4,5,6}
    \draw (1pt,\y cm) -- (-1pt,\y cm) node[anchor=east] {$\y$};
    \draw[line width=0.1cm] (1,7) -- (1,4)--(3,4)--(3,2)--(5,2)--(5,1)--(7,1);
    \draw[line width=0.1cm] (2,7)--(2,5)--(4,5)--(4,3)--(7,3);
    \draw[line width=0.1cm] (6,7)--(6,6)--(7,6);
					\fill[black] (1,4) circle (.2cm);
					\fill[black] (2,5) circle (.2cm);
					\fill[black] (3,2) circle (.2cm);
					\fill[black] (4,3) circle (.2cm);
					\fill[black] (5,1) circle (.2cm);
					\fill[black] (6,6) circle (.2cm);
					\draw [fill] (2.7,3.7) rectangle (3.3,4.3);
					\draw [fill] (4.7,1.7) rectangle (5.3,2.3);
					\draw [fill] (3.7,4.7) rectangle (4.3,5.3);
\end{tikzpicture}}&
\scalebox{0.8}{\begin{tikzpicture}
\draw[step=1cm,gray,very thin] (0,0) grid (7,7);
\foreach \x in {1,2,3,4,5,6}
    \draw (\x cm,1pt) -- (\x cm,-1pt) node[anchor=north] {$\x$};
\foreach \y in {1,2,3,4,5,6}
    \draw (1pt,\y cm) -- (-1pt,\y cm) node[anchor=east] {$\y$};
    \draw[line width=0.1cm] (3,7) -- (3,4)--(5,4)--(5,2)--(7,2);
    \draw[line width=0.1cm] (4,7)--(4,5)--(7,5);
					\fill[black] (3,4) circle (.2cm);
					\fill[black] (5,2) circle (.2cm);
					\fill[black] (4,5) circle (.2cm);
\end{tikzpicture}}
\end{tabular}
\end{center}
\caption{Iterating the shadow line procedure to obtain the second row of $P(\pi)$ for $\pi=452316$}
\label{F:shadow2}
\end{figure}

Next, we characterize the permutations $\pi$ for which $\pi \pi^r \in \mathcal{R}_{\binom{k}{2}}(12\cdots k)$ in terms of shadow lines.

\begin{lemma}\label{L:labels}
Suppose $\pi \in \mathcal{S}_{\binom{k}{2}}$. Then $\pi\pi^r$ avoids $12 \cdots k$ if and only if 
\begin{enumerate}
\item shadow line $i$ of $\pi$ contains $k-i$ points for all $1 \leq i \leq k-1$, and 
\item the $j$th point on the $i$th shadow line of $\pi$ is the unique point such that the longest increasing subsequence ending in that point has length $i$ and the longest decreasing subsequence ending in that point has length $j$.
\end{enumerate}
\end{lemma}

\begin{proof}
Suppose $\pi \in \mathcal{S}_{\binom{k}{2}}$ and $\pi \pi^r$ avoids $12\cdots k$.  Using the labeling from Theorem \ref{T:erdos} there is a unique point in $\pi$ with each label $(a,b)$ for $2 \leq a+b \leq k$.   For any permutation $\pi$, by construction, each point on shadow line $i$ has $a=i$.  Moreover, the $j$th point on shadow line $i$ has $b \geq j$.  Since we know there is exactly one point with each label, it must be the case that the $j$th point on shadow line $i$ has label $(i,j)$.  This labeling also results in exactly $k-i$ points on the $i$th shadow line, as desired.

Suppose, on the other hand that $\pi \pi^r$ contains $12 \cdots k$.  This means that the labeling from Theorem \ref{T:erdos} results in a point $\pi_{\ell}$ with the label $(a,b)$ where $a+b \geq k+1$.  By the shadow line construction, this point must be on shadow line $a$.  If $\pi_{\ell}$ is really the $b$th point on shadow line $a$, then we have violated condition (1) since shadow line $a$ has at least $b \geq k+1-a$ points.  If $\pi_{\ell}$ is not the $b$th point on shadow line $a$, then we have violated condition (2).
\end{proof}

As in Viennot's general case, these shadow lines are a tool to better understand the structure of $P(\pi)$ in the context of $\mathcal{R}_{\binom{k}{2}}(12\cdots k)$.

\begin{lemma}\label{L:shadow}
If $\pi\pi^r$ avoids $12 \cdots k$, then column $i$ of $P(\pi)$ consists of the entries from the $i$th shadow line of $\pi$.  
\end{lemma}

\begin{proof}

By Lemma \ref{L:labels} we know that the $j$th point on the $i$th shadow line of $\pi$ is the unique point such that the longest increasing subsequence ending in that point has length $i$ and the longest decreasing subsequence ending in that point has length $j$ and that the $i$th shadow line of $\pi$ has $k-i$ points for all $1 \leq i \leq k-1$.

We claim that if two points are adjacent points on the same shadow line at one iteration of the shadow line construction, then they must be on the same shadow line at the next iteration.  Suppose to the contrary that $\pi_a>\pi_b$ are adjacent points on the same shadow line in one iteration, but they are on different shadow lines in the next iteration.   Then one of the two cases in Figure \ref{F:shadow3} must occur.  That is, when we consider the square points at height $\pi_a$ and $\pi_b$, either there is a square on a different shadow line that forms the $1$ in a $312$ pattern (in which case $\pi_a$ moves to an earlier shadow line than $\pi_b$), or there is a square on a different shadow line that forms the $2$ in a $231$ pattern (in which case $\pi_b$ moves to an earlier shadow line than $\pi_a$), but not both (in which case, $\pi_a$ and $\pi_b$ would stay on the same line). In either case, since both $\pi_a$ and $\pi_b$ are involved in the next iteration, there must be at least one more point $\pi_c$ further to the right on their shadow line. 

In the case where there is a square on a different shadow line that forms the $1$ in a 312 pattern (on the left of Figure \ref{F:shadow3}), suppose that $\pi_d$ is the height of the square that plays the role of this 1.  Call the next point on $\pi_d$'s shadow line $\pi_e$, Since $\pi_d$'s square forms a 312 pattern with $\pi_a$ and $\pi_b$'s squares, $\pi_e$ appears horizontally between $\pi_b$ and $\pi_c$.  If $\pi_d$ is the first point on the shadow line, we have a contradiction since the longest decreasing subsequence ending at $\pi_e$ should have length 2.  However, $\pi_a$, $\pi_b$, and $\pi_e$ form a longer decreasing subsequence ending in $\pi_e$.  If $\pi_d$ is not the first point on its shadow line, then call the previous point $\pi_f$.  Since there is no 231 pattern using the squares from $\pi_a$ and $\pi_b$, $\pi_f>\pi_a$.  The longest decreasing subsequence ending in $\pi_e$ should be the one that follows the shadow line containing $\pi_e$, ending in $\pi_f \pi_d \pi_e$, however, taking the same decreasing subsequence and replacing $\pi_d \pi_e$ with $\pi_a\pi_b\pi_e$ forms an even longer decreasing subsequence, which is a contradiction.  So, this case is impossible if $\pi \pi^r$ avoids $12\cdots k$.  

The case where there is a square on a different shadow line that forms the 2 in a 231 pattern (on the right of Figure \ref{F:shadow3}) is similar.  Again, suppose $\pi_d$ is the height of the square that plays the role of this 2 and call the next point on $\pi_d$'s shadow line $\pi_e$.  Since $\pi_d$'s square forms a 231 pattern with $\pi_a$ and $\pi_b$'s squares, $\pi_e$ appears before $\pi_b$.  If $\pi_e$ is the last element on its shadow line we have a contradiction because $\pi_e$'s shadow line should have one more element than $\pi_a$'s shadow line, and the longest decreasing subsequence ending in $\pi_e$ (following its shadow line) should be longer than the longest decreasing subsequence ending in $\pi_c$.  However, taking $\pi_e$'s shadow line and replacing $\pi_d \pi_e$ with $\pi_d \pi_b \pi_c$ produces a longer decreasing subsequence ending in $\pi_c$.  If $\pi_e$ is not the last element on its shadow line, then call the next point on the shadow line $\pi_f$.  Since there is no 312 pattern using the squares from $\pi_a$ and $\pi_b$, we know that $\pi_f$ appears to the right of $\pi_c$.  Again, the longest decreasing subsequence ending in $\pi_f$ should be the decreasing subsequence formed by following $\pi_f$'s shadow line.  However, following this shadow line and replacing $\pi_d \pi_e \pi_f$ with $\pi_d \pi_b\pi_c\pi_f$ forms a longer decreasing subsequence ending in $\pi_f$, which is a contradiction.  So, this case is also impossible if $\pi \pi^r$ avoids $12\cdots k$.  

In summary, if $\pi \pi^r$ avoids $12\cdots k$, then two adjacent elements on the same shadow line in one iteration of the shadow lines construction will be adjacent elements on the same shadow line at the next iteration.  This means that each row of $P(\pi)$ takes one element from each of the original shadow lines and the $k-i$ elements of the $i$th shadow line appear in the $i$th column of $P(\pi)$.
\end{proof}

\begin{figure}

\begin{center}
\begin{tabular}{cc}
\scalebox{0.8}{
\begin{tikzpicture}
\draw[step=1cm,gray,very thin] (0,0) grid (7,7);
    \draw[line width=0.1cm] (3,7)--(3,5)--(4,5)--(4,4)--(6,4)--(6,3)--(7,3);
		    \draw[line width=0.1cm] (2,6)--(2,2)--(5,2)--(5,1)--(7,1);
				\draw[dashed,line width=0.1cm] (1,7)--(1,6)--(2,6);
					\fill[black] (3,5) circle (.1cm) node[align=left,   below] {$\pi_a$};
					\fill[black] (4,4) circle (.1cm) node[align=left,   below] {$\pi_b$};
					\fill[black] (6,3) circle (.1cm) node[align=left,   below] {$\pi_c$};
					\fill[black] (2,2) circle (.1cm) node[align=left,   below] {$\pi_d$};
					\fill[black] (5,1) circle (.1cm) node[align=left,   below] {$\pi_e$};
					\fill[black] (1,6) circle (.1cm) node[align=left,   below] {$\pi_f$};
					\draw [fill] (3.8,4.8) rectangle (4.2,5.2);
					\draw [fill] (5.8,3.8) rectangle (6.2,4.2);					
					\draw [fill] (4.8,1.8) rectangle (5.2,2.2);
\end{tikzpicture}}&
\scalebox{0.8}{\begin{tikzpicture}
\draw[step=1cm,gray,very thin] (0,0) grid (7,7);
    \draw[line width=0.1cm] (3,7)--(3,6)--(4,6)--(4,4)--(5,4)--(5,3)--(7,3);
		    \draw[line width=0.1cm] (1,7)--(1,5)--(2,5)--(2,2)--(6,2);
				\draw[dashed,line width=0.1cm] (6,2)--(6,1)--(7,1);
					\fill[black] (3,6) circle (.1cm) node[align=left,   below] {$\pi_a$};
					\fill[black] (4,4) circle (.1cm) node[align=left,   below] {$\pi_b$};
					\fill[black] (5,3) circle (.1cm) node[align=left,   below] {$\pi_c$};
					\fill[black] (1,5) circle (.1cm) node[align=left,   below] {$\pi_d$};
					\fill[black] (2,2) circle (.1cm) node[align=left,   below] {$\pi_e$};
					\fill[black] (6,1) circle (.1cm) node[align=left,   below] {$\pi_f$};
					\draw [fill] (3.8,5.8) rectangle (4.2,6.2);
					\draw [fill] (4.8,3.8) rectangle (5.2,4.2);					
					\draw [fill] (1.8,4.8) rectangle (2.2,5.2);
\end{tikzpicture}}
\end{tabular}
\end{center}
\caption{Cases for points on the same shadow line}
\label{F:shadow3}
\end{figure}

Note that the converse of Lemma \ref{L:shadow} is false.  For example, when $\pi = 645123$, then $P(\pi)$ has shape (3,2,1) where the columns of $P(\pi)$ correspond to the original shadow lines of $\pi$.  However, $\pi \pi^r$ contains 1234 using the digits 1, 2, and 3 from $\pi$ along with the digit 4 from $\pi^r$.

While Lemma \ref{L:labels} completely characterizes $\pi$ for which $\pi \pi^r \in \mathcal{R}_{\binom{k}{2}}(12\cdots k)$ by giving conditions on shadow lines, we have only given partial conditions on the correponding tableau $P(\pi)$.  By Lemma \ref{L:staircase}, we know that it is necessary for $P(\pi)$ to have shape $(k-1, k-2,\dots, 1)$ and by Lemma \ref{L:shadow} is it necessary that the $i$th column of $P(\pi)$ consist of the entries of the $i$th shadow line of $\pi$.  As we have seen, neither of these conditions is sufficient for $\pi \pi^r \in \mathcal{R}_{\binom{k}{2}}(12\cdots k)$.  The missing condition is that given in Theorem \ref{T:syt}; i.e. the diagonals of $P(\pi)$ must be increasing.

\begin{proof}[Proof of Theorem \ref{T:syt}]

By Lemma \ref{L:staircase} and Lemma \ref{L:shadow} we restrict our attention to $\pi \in \mathcal{S}^*_{\binom{k}{2}}$ such that the $i$th column of $P(\pi)$ consists of the entries of the $i$th shadow line of $\pi$. We show that the diagonals of $P(\pi)$ are increasing if and only if $\pi \pi^r$ avoids $12 \cdots k$.  

Suppose that $d$ is the smallest integer for which there is a decrease in a diagonal between some element $x$ in column $d$ and element $y$ in column $d+1$.  Since $x>y$ and $x$ is on an earlier shadow line than $y$, we know that $x$ appears to the left of $y$.  The entry $z$ in row 1 column $d+1$ corresponds to the last element on the $(d+1)$st shadow line, which means the longest decreasing subsequence in $\pi$ ending in $z$ should have length $k-d-1$.  However, the elements $x$ through the end of column $d$ together with elements $y$ through $z$ of column $d+1$ form a decreasing subsequence of length $k-d$ in $\pi$ that ends in $z$, which contradicts $\pi \pi^r$ avoiding $12\cdots k$.

Suppose on the other hand that all diagonals in $P(\pi)$ are increasing.  As in the proof of Theorem \ref{T:erdos}, given (partial) permutation $\pi^*$, let $a_k$ be the length of the longest increasing subsequence of $\pi^*$ ending in $\pi^*_k$ and let $b_k$ be the length of the longest decreasing subsequence of $\pi^*$ ending in $\pi^*_k$.  The (partial) permutation under consideration will be made clear from context.  We claim that for any element $\pi_k$, if $\pi_k$ is the $j$th element of the $i$th shadow line of $\pi$ then $(a_k,b_k)=(i, j)$.  Since $\pi_k$ is on the $i$th shadow line, $a_k=i$ automatically, and we need only check that the maximal decreasing sequence ending in this entry is of length $j$.

Let $\pi^{(f)}$ be the partial permutation formed by the first $f$ shadow lines of $\pi$.  We claim that that $\pi^{(f)}$ uses labels $(a,b)$ where $1 \leq a \leq f$ and $2 \leq a+b \leq k$ each exactly once and proceed by induction on $f$.  Notice that $\pi^{(1)}$ consists of only the first shadow line.  Its $j$th digit is at the end of a maximal increasing sequence of length 1 and a maximal decreasing sequence of length $j$, and $\pi^{(1)}$ has length $k-1$, so it satisfies our claim.

Now suppose that the partial permutation $\pi^{(f-1)}$ uses labels $(a,b)$ where $1 \leq a \leq f-1$ and $2 \leq a+b \leq k$ each exactly once.  We will show that $\pi^{(f)}$ uses labels $(a,b)$ where $1 \leq a \leq f$ and $2 \leq a+b \leq k$ each exactly once.

Since the increasing diagonals property implies that the $j$th point on shadow line $f$ is larger than the $j$th point on any previous shadow line, that point cannot be at the end of a longer decreasing sequence in $\pi^{(f)}$ than the one it already ends just by following shadow line $f$.  

However, can including the digits from shadow line $f$ affect the longest decreasing subsequence ending at some point in $\pi^{(f-1)}$?  We claim not.  Suppose, to the contrary that there is some point $\pi_k$ in $\pi^{(f)}$ that has the expected label $(a_k,b_k)=(i,j)$ as the $j$th element of the $i$th shadow line when we restrict to $\pi^{(f-1)}$, but its second coordinate is larger when we restrict to $\pi^{(f)}$.  

This implies that there is some $j^*$ for which the $j^*$th element of shadow line $f$ appears before the $j^*$th element of an earlier shadow line $e$ in $\pi^{(f)}$ in order to make a longer decreasing subsequence ending in $\pi_k$.  However, we assumed that column $i$ of $P(\pi)$ corresponds to shadow line $i$ of $\pi$ for all $i$, which means each entry stays in the same column of $P(\pi)$ as it was originally inserted throughout the Robinson-Schensted bumping algorithm.  When there are $j^*$ elements of shadow line $f$ before the $j^*$th element of shadow line $e$, one of the elements from column $f$ will necessarily be bumped to an earlier column during the Robinson-Schensted bumping algorithm.  So this is impossible.  That is, given that $\pi^{(f-1)}$ uses labels $(a,b)$ where $1 \leq a \leq f-1$ and $2 \leq a+b \leq k$ each exactly once, if $P(\pi)$ satisfies the increasing diagonals property then $\pi^{(f)}$ uses labels $(a,b)$ where $1 \leq a \leq f$ and $2 \leq a+b \leq k$ each exactly once.  Since this is true for all $1 \leq f \leq k-1$, when we consider the labels of $\pi = \pi^{(k-1)}$ we see that $\pi$ avoids $12 \cdots k$ by Lemma \ref{L:labels}.

\end{proof}

We are now in a position to compute $\mathrm{r}_{\binom{k}{2}}(12\cdots k)$ since $\mathcal{R}_{\binom{k}{2}}(12\cdots k)$ is in bijection with pairs of standard Young tableau $(P,Q)$ where $P$ and $Q$ both have shape $(k-1, k-2, \dots, 1)$ and $P$ has increasing diagonals.

\begin{corollary}
$$\mathrm{r}_{\binom{k}{2}}(12\cdots k) = \left(\left(\dfrac{1}{2}(k-1)^2+\dfrac{k}{2}-\dfrac{1}{2}\right)!\prod_{i=1}^{k-1}\dfrac{(i-1)!}{(2i-1)!}\right)\left(\dfrac{\binom{k}{2}!}{\prod_{i=1}^{k-1}\left(2i-1\right)^{k-i}}\right).$$
\label{C:rsk}
\end{corollary}

\begin{proof}
Here, the first factor of $$\left(\left(\dfrac{1}{2}(k-1)^2+\dfrac{k}{2}-\dfrac{1}{2}\right)!\prod_{i=1}^{k-1}\dfrac{(i-1)!}{(2i-1)!}\right)$$ is OEIS sequence A003121 in \cite{OEIS}, which is the number of standard Young tableaux of shape $(k-1, \dots, 1)$ with increasing diagonals, while $$\left(\dfrac{\binom{k}{2}!}{\prod_{i=1}^{k-1}\left(2i-1\right)^{k-i}}\right)$$ is the total number of standard Young tableaux of shape $(k-1, \dots, 1)$, as computed by the hook length formula and described in OEIS sequence A005118.
\end{proof}

\section{Avoiding patterns of small length}\label{S:small}

We have already described reverse double lists avoiding a monotone pattern of arbitrary length.  Although $\mathrm{r}_n(12\cdots k)=0$ for sufficiently large $n$, there are other patterns $\rho$ for which $\mathrm{r}_n(\rho)$ exhibits other behavior.  In the rest of the paper, we consider reverse double lists avoiding a variety of non-monotone patterns. 

\subsection{Avoiding patterns of length 3}\label{S:three}

In this subsection, we consider patterns of length 3.  First, notice that the graph of a reverse double list $\sigma \in \mathcal{R}_n$ is a set of points on the rectangle $[0,2n+1] \times [0,n+1]$.  Using the reverse and complement involutions described in Section \ref{S:Intro}, $$\sigma \in \mathcal{R}_n(\rho) \Longleftrightarrow \sigma^r \in \mathcal{R}_n(\rho^r) \Longleftrightarrow \sigma^c \in \mathcal{R}_n(\rho^c).$$  In fact, if $\sigma \in \mathcal{R}_n(\rho)$, then $\sigma=\sigma^r$, so $\mathcal{R}_n(\rho)=\mathcal{R}_n(\rho^r)$.  Partition the set of permutation patterns of length $k$ into equivalence classes where $\rho \sim \tau$ means that $\mathrm{r}_n(\rho) = \mathrm{r}_n(\tau)$ for $n \geq 1$.  When $\rho \sim \tau$, $\rho$ and $\tau$ are said to be \emph{Wilf equivalent}.  When this equivalence holds because of one of the symmetries of the rectangle, we say that $\rho$ and $\tau$ are \emph{trivially Wilf equivalent}.  Using trivial Wilf equivalence we have that $123 \sim 321$ and $132 \sim 213 \sim 231 \sim 312$, so we need only consider 2 patterns in this section: 123 and 132.

With pattern-avoiding permutations, avoiding a pattern of length 3 is the first non-trivial enumeration, and for any pattern $\rho$ of length 3, we have that $\mathrm{s}_n(\rho)$ is the $n$th Catalan number.  Reverse double lists are more restrictive, so we obtain simpler sequences for $\mathrm{r}_n(\rho)$.  More strikingly, although $\mathrm{s}_n(123)=\mathrm{s}_n(132)$ for $n \geq 1$, we obtain two distinct sequences in this new context.

By Theorem \ref{T:erdos}, $\mathrm{r}_n(123)=0$ for $n \geq 4$.  On the other hand, there are 132-avoiders of arbitrary length.

\begin{prop}
$\mathrm{r}_n(132)=\mathrm{r}_n(213)=\mathrm{r}_n(231)=\mathrm{r}_n(312)=2$ for $n \geq 2$.
\label{T:132}
\end{prop}

\begin{proof}
The $n=2$ case is straightforward to check.

Now, suppose $n \geq 3$.  We claim $$\mathcal{R}_n(132)=\left\{n(n-1)\cdots 312213\cdots (n-1)n, n(n-1)\cdots 321123\cdots (n-1)n\right\}.$$

Assume for contradiction that $\sigma=\pi\pi^r \in \mathcal{R}_n(132)$ and $\pi_1\neq n$. This implies that either $\pi_1=1$ or $2\leq\pi_1\leq n-1$.  If $\pi_1=1$, then the digits 1 and 3 in $\pi$ and $2$ in $\pi^r$ form a 132 pattern.  If $2\leq\pi_1\leq n-1$, then the digit 1 in $\pi$ together with $n$ and $\pi_1$ in $\pi^r$ form a 132 pattern. Thus, $\pi_1=n$.  Further, if $\sigma \in \mathcal{R}_n(132)$, then let $\sigma^{\prime}$ be the reverse double list formed by deleting both copies of $n$.  It must be the case that $\sigma^{\prime} \in \mathcal{R}_{n-1}(132)$ since a copy of 132 in $\sigma^{\prime}$ would also be a copy of 132 in $\sigma$.  Thus, $\mathcal{R}_n(132)=\left\{n\sigma^{\prime}n \middle| \sigma^{\prime} \in \mathcal{R}_{n-1}(132)\right\}$.  Since $\mathcal{R}_2(132)=\{1221,2112\}$, the claim follows by induction, and so $|\mathcal{R}_n(132)|=2$ when $n\geq 2$.  
\end{proof}

At this point, we have completely characterized reverse double lists avoiding a single pattern of length 3.  Although we obtained only trivial sequences, the fact that we obtained two distinct Wilf classes when avoiding a pattern of length 3 mirrors results for pattern-avoiding double lists in \cite{VERUM2014}.

\subsection{Avoiding patterns of length 4}\label{S:four}

Next, we analyze reverse double lists avoiding a single pattern of length 4.  Using the symmetries of the rectangle, we can partition the 24 patterns of length 4 into 7 trivial Wilf classes, as shown in Table \ref{T:length4}.  There is one non-trivial Wilf equivalence for patterns of length 4; namely $1324 \sim 2143$.  To contrast: for double lists, there are no non-trivial Wilf equivalences for patterns of length 4.  For permutations, we have an additional trivial Wilf equivalence since $\mathrm{s}_n(\rho) = \mathrm{s}_n(\rho^{-1})$ for $n \geq 1$, so $\mathrm{s}_n(1342)=\mathrm{s}_n(1423)$.  There are a number of additional non-trivial Wilf equivalences for pattern-avoiding permutations so that every length 4 pattern is equivalent to one of 1342, 1234, or 1324.  For large $n$, we have that $$\mathrm{s}_n(1342^\bullet) < \mathrm{s}_n(1234^\dagger) < \mathrm{s}_n(1324^\circ).$$  In Table \ref{T:length4} each pattern is marked according to its Wilf equivalence class for permutations; patterns equivalent to 1342 are marked with $\bullet$, those equivalent to 1234 are marked with $\dagger$, and those equivalent to 1324 are marked with $\circ$.  For permutations, the monotone pattern 1234 is neither the hardest nor the easiest pattern to avoid; for double lists, it is the easiest pattern to avoid, and for reverse double lists it is the hardest pattern to avoid.  Other than the trivial equivalences of reverse and complement, Wilf equivalence in the context of reverse double lists appears to be a very different phenomenon than equivalence in the contexts of permutations or double lists.  We now consider each of these patterns in turn.

\begin{table}[hbt]
\begin{center}
\begin{tabular}{|l|l|}
\hline
Pattern $\rho$& $\left\{\mathrm{r}_n(\rho)\right\}_{1 \leq n \leq 9}$\\
\hline
$1234^\dagger \sim 4321^\dagger$ & 1, 2, 6, 16, 32, 32, 0, 0, 0 \\
\hline
$1243^\dagger  \sim 2134^\dagger  \sim 3421^\dagger  \sim 4312^\dagger$ & 1, 2, 6, 16, 34, 62, 102, 156, 226\\
\hline
$1324^\circ \sim 4231^\circ$&1, 2, 6, 16, 36, 76, 156, 316, 636\\
\hline
$2143^\dagger \sim 3412^\dagger$&1, 2, 6, 16, 36, 76, 156, 316, 636 \\
\hline
$1423^\bullet \sim 2314^\bullet \sim 3241^\bullet \sim 4132^\bullet$ & 1, 2, 6, 16, 36, 80, 178, 394, 870 \\
\hline
$1432^\dagger  \sim 2341^\dagger  \sim 3214^\dagger  \sim 4123^\dagger $ &1, 2, 6, 16, 38, 92, 222, 536, 1294\\
\hline
$1342^\bullet \sim 2431^\bullet \sim 3124^\bullet \sim 4213^\bullet$ & 1, 2, 6, 16, 40, 98, 238, 576, 1392\\
\hline
$2413^\bullet \sim 3142^\bullet$&1, 2, 6, 16, 44, 120, 328, 896, 2448\\
\hline
\end{tabular}
\end{center}
\caption{Enumeration of reverse double lists avoiding a pattern of length 4}
\label{T:length4}
\end{table}

\subsubsection{The pattern 1234}

By Theorem \ref{T:erdos}, $\mathrm{r}_n(1234)=0$ for $n \geq 7$.

\subsubsection{The pattern 1243}

\begin{theorem}
$\mathrm{r}_n(1243)=\frac{n^3}{3}-\frac{7n}{3}+4$ for $n \geq 2$.
\label{T:1243}
\end{theorem}

\begin{proof}
We claim that $$\mathrm{r}_n(1243)=\begin{cases}
	n! & n\leq 2\\
	\mathrm{r}_{n-1}(1243)+	(n+1)(n-2) & n\geq 3
	\end{cases}.$$

The cases where $n \leq 2$ are easy to check by brute force, so we focus on the case where $n \geq 3$.  Let $\sigma=\pi\pi^r \in \mathcal{R}_n(1243)$, and let $\sigma^{\prime} \in \mathcal{R}_{n-1}(1243)$ be the reverse double list formed by deleting both copies of $n$ in $\sigma$.  We consider 4 cases based on the value of $\pi_1$.

Suppose $\pi_1=n$.  Then $\sigma = n \sigma^{\prime}n$.  Since $n$ can only play the role of a $4$ in a $1243$ pattern, $n\sigma^{\prime}n \in \mathcal{R}_n(1243)$ for any $\sigma^{\prime} \in \mathcal{R}_{n-1}(1243)$.  There are $\mathrm{r}_{n-1}(1243)$ reverse double lists on $n$ letters in this case.

Suppose $\pi_1=n-1$.  We claim that $\pi_2=n-2$.  Assume for contradiction that $\pi_2\neq n-2$.  This implies $\pi_2=n$ or $1 \leq \pi_2 \leq n-3$. If $\pi_2=n$, then the digit 1 in $\pi$ along with the digits $n-2$, $n$, and $n-1$ in $\pi^r$ form a 1243 pattern. If $1 \leq \pi_2 \leq n-3$, then the digits $\pi_2$ and $n-2$ in $\pi$ along with the digits $n$ and $n-1$ in $\pi^r$ form a 1243 pattern. Thus, if $\pi_1=n-1$, $\pi_2=n-2$.  By a similar argument, $\pi_3=n-3$, and in general, $\pi_i=n-i$ for $1 \leq i \leq n-2$.  Finally, $\pi_{n-1}=1$ and $\pi_n=n$ or $\pi_{n-1}=n$ and $\pi_n=1$.  Thus there are exactly two 1243-avoiding reverse double lists on $n$ letters in this case.

Suppose $\pi_1=1$.  We claim that $\pi_2=n$.  Assume for contradiction that $\pi_2\neq n$. This implies $2 \leq \pi_2 \leq n-2$ or $\pi_2=n-1$. If $2 \leq \pi_2 \leq n-2$, then the digits 1, $\pi_2$ and $n$ in $\pi$ along with the digit $n-1$ in $\pi^r$ form a 1243 pattern. If $\pi_2=n-1$, then the digits 1 and 2 in $\pi$ along with the digits $n$ and $n-1$ in $\pi^r$ form a 1243 pattern. Hence, if $\pi_1=1$, then $\pi_2=n$.  By a similar argument, $\pi_3=n-1$, and in general, $\pi_i=n+2-i$ for $2 \leq i \leq n-2$.  Finally, either $\pi_{n-1}=2$ and $\pi_n=3$ or $\pi_{n-1}=3$ and $\pi_n=2$.  Thus there are exactly two 1243-avoiding reverse double lists on $n$ letters in this case.

Finally, suppose $2 \leq \pi_1 \leq n-2$.  Let $\pi_1=a$.  First, we claim that $\pi_1\cdots \pi_{a-1}=a(a-1)\cdots 2$.  If $\pi_1=2$, this claim is already true, so suppose $\pi_1>2$ but $\pi_2 \neq a-1$.  If $\pi_2<a-1$ then the digits $\pi_2$ and $a-1$ in $\pi$ along with the digits $n$ and $a$ in $\pi^r$ form a 1243 pattern.  If $\pi_2>a$ then the digit 1 in $\pi$ along with the digits 2, $\pi_2$, and $a$ in $\pi^r$ form a 1243 pattern.  Therefore, $\pi_2=a-1$.  A similar argument shows that $\pi_j=a-j+1$ for $1 \leq j \leq a-1$.

We have shown that $\pi_{a-1}=2$.  Now, consider $\pi_a$.  We claim that $\pi_a \in \{1,n\}$.  Suppose to the contrary that $a+1 \leq \pi_a \leq n-1$.  If $a+1 \leq \pi_a \leq n-2$, then the digits 2, $\pi_a$, and $n$ in $\pi$ along with $n-1$ in $\pi^r$ form a 1243 pattern.  If $\pi_a=n-1$ then the digits 2 and $a+1$ in $\pi$ along with $n$ and $n-1$ in $\pi^r$ form a 1243 pattern.  Therefore, either $\pi_a=1$ or $\pi_a=n$. 

By a similar argument, there are at most 2 possible values for $\pi_j$ where $a+1\leq j\leq n-2$.  Either $\pi_j=1$ or $\pi_j = \max(\{1,\dots,n\} \setminus \{\pi_1,\dots, \pi_{j-1}\})$.  Finally, $a+2$ and $a+1$ may appear in either order in $\pi$. Hence, the number $1$ can take position $h$ where $a \leq h \leq n$, and the remaining positions of $\pi_a \cdots \pi_n$ must be filled with either $n (n-1)  \cdots (a+3)( a+2)( a+1)$ or $n( n-1) \cdots( a+3)( a+1)( a+2)$. 

To summarize, if $2 \leq \pi_1 \leq n-2$, where $\pi_1=a$, then we have $\pi_1\cdots \pi_{a-1}=a\cdots 2$.  The digit 1 may appear in any of the $n-a+1$ remaining positions.  After the initial $a-1$ digits and the location of $1$ are chosen, there are 2 ways to fill in the rest of $\pi$.  Either the remaining digits are $n (n-1)  \cdots (a+3)( a+2)( a+1)$ or $n( n-1) \cdots( a+3)( a+1)( a+2)$.   There are $\sum_{a=2}^{n-2} 2(n-a+1)=n^2-n-6$ 1243-avoiding reverse double lists on $n$ letters in this case.

Combining all four cases, we see that $$\mathrm{r}_n(1243)=\mathrm{r}_{n-1}(1243) + 2 + 2 + n^2-n-6 = \mathrm{r}_{n-1}(1243)+(n+1)(n-2).$$

The theorem then follows from the facts that (i) the $n=2$ terms of the theorem statement and the claim agree and (ii) $$\left(\frac{n^3}{3}-\frac{7n}{3}+4\right) - \left(\frac{(n-1)^3}{3}-\frac{7(n-1)}{3}+4\right) = (n+1)(n-2).$$
\end{proof}

\subsubsection{The patterns 1324 and 2143}

We now come to two patterns that are Wilf-equivalent for non-trivial reasons.  Although there is not a simple symmetry of graphs that demonstrates $\mathrm{r}_n(1324)=\mathrm{r}_n(2143)$, we show that $\mathrm{r}_n(1324)$ and $\mathrm{r}_n(2143)$ each satisfy the same recurrence.

\begin{theorem}
$\mathrm{r}_n(1324)=5\cdot 2^{n-2}-4$ for $n \geq 3$.
\label{T:1324}
\end{theorem}

\begin{proof}
The case where $n = 3$ is easy to check by brute force, so we focus on when $n \geq 4$.  Let $\sigma=\pi\pi^r \in \mathcal{R}_n(1324)$. We consider 2 cases.

Suppose $\sigma^{\prime}  \in \mathcal{R}_{n-1}(1324)$ with $\sigma^{\prime}_1=a$.  We can build $\sigma =\pi\pi^r \in \mathcal{R}_n(1324)$ by either using $\sigma_1=a$ or $\sigma_1=a+1$, and incrementing all digits of $\sigma^{\prime}$ that are at least $\sigma_1$ by 1.  It is impossible for $\sigma_1$ to participate in a 1324 pattern with $\sigma_2$ since they are consecutive integers.  It is impossible for $\sigma_1$ to participate in a 1324 pattern without $\sigma_2$ since any 1324 pattern involving $\sigma_1$ would imply the existence of a 1324 pattern using $\sigma_2$ in place of $\sigma_1$.  Therefore, there are $2\mathrm{r}_{n-1}(1324)$ reverse double lists where $\sigma_2=\sigma_1 \pm 1$.

Now, suppose that $\left|\pi_2-\pi_1\right|>1$.  Then if $\pi_1<\pi_2<n$, the digits $\pi_1$, $\pi_2$, and $\pi_1+1$ in $\pi$ along with the digit $n$ in $\pi^r$ form a 1324 pattern.  If $\pi_1>\pi_2>1$, then the digit 1 in $\pi$ along with the digits $\pi_1-1$, $\pi_2$, and $\pi_1$ in $\pi^r$ form a 1324 pattern.  So, if $\left|\pi_2-\pi_1\right|>1$ it must be the case that either $\pi_1<\pi_2=n$ or $\pi_1>\pi_2=1$.

If $\pi_1<\pi_2=n$, then $\pi_1=n-2$, otherwise $\pi_1$ and $\pi_1+2$ in $\pi$ and $\pi_1+1$ and $n$ in $\pi^r$ form a 1324 pattern.  Also, the word $\pi_3\cdots \pi_n\pi_n \cdots \pi_3$ must avoid 132, or $n\pi_3\cdots \pi_n\pi_n \cdots \pi_3n$ will contain a 1324 pattern.  We know that there are exactly two 132-avoiding reverse double lists for $n \geq 2$, so there are two 1324-avoiding reverse double lists on $n$ letters where $\pi_1<\pi_2=n$.  Similarly, by taking the complements of all reverse double lists where $\pi_1<\pi_2=n$ we see that there are two reverse double lists on $n$ letters where $\pi_1>\pi_2=1$.

In summary, given $\sigma^{\prime} \in \mathcal{R}_{n-1}(1324)$, we can produce exactly two members $\sigma=\pi\pi^r$ of $\mathcal{R}_{n}(1324)$ where $\left|\pi_1-\pi_2\right|=1$.  There are 4 additional reverse double lists where $\left|\pi_1-\pi_2\right|=2$, so for $n \geq 4$, we have 

$$\mathrm{r}_{n}(1324)=2\mathrm{r}_{n-1}(1324)+4.$$

From this recurrence it follows that $\mathrm{r}_{n}(1324) = 5\cdot 2^{n-2}-4$ for $n \geq 3$.
\end{proof}

Reverse double lists avoiding 2143 follow the same recurrence but for different structural reasons.

\begin{theorem}
$\mathrm{r}_n(2143)=5\cdot 2^{n-2}-4$ for $n\geq 3$.
\label{T:2143}
\end{theorem}

\begin{proof}
The case where $n = 3$ is easy to check by brute force, so we focus on when $n \geq 4$.  Let $\sigma=\pi\pi^r \in \mathcal{R}_n(2143)$. 

Suppose $\sigma^{\prime} \in \mathcal{R}_{n-1}(2143)$.  We can build $\sigma =\pi\pi^r \in \mathcal{R}_n(2143)$ by either using $\sigma_1=1$ and incrementing all digits of $\sigma^{\prime}$ or by using $\sigma_1=n$.  In both cases is impossible for $\sigma_1$ to participate in a 2143 pattern since $\sigma_1$ is either the largest or the smallest digit in $\sigma$.  Therefore, there are $\mathrm{r}_{n-1}(2143)$ reverse double lists where $\sigma_1=1$ and $\mathrm{r}_{n-1}(2143)$ reverse double lists where $\sigma_1=n$.

Now, suppose $1<\pi_1<n$.  Then, every digit larger than $\pi_1$ must be in increasing order in $\pi^r$, otherwise $\pi_1$ and 1 in $\pi$ together with the decreasing pair of larger digits in $\pi^r$ form a 2143 pattern.  Similarly, every digit smaller than $\pi_1$ must appear in increasing order in $\pi$, otherwise the decreasing pair of smaller digits in $\pi$ together with $n$ and $\pi_1$ in $\pi^r$ form a 2143 pattern.

If $3 \leq \pi_1 \leq n-2$, we already have a problem since either 1 appears before $n$ in $\pi$ or 1 appears before $n$ in $\pi^r$.  In the first case, $\pi_1 1 n(n-1)$ is a copy of 2143 in $\sigma$, and in the second case, $21n\pi_1$ is a copy of 2143 in $\sigma$.

So it must be the case that $\pi_1=2$ or $\pi_1=n-1$.  In the first case, either $\pi_{n-1}=1$ or $\pi_n=1$ and all other digits of $\pi$ are in decreasing order.  In the second case, we take the complement of the words where $\pi_1=2$ to get the words where $\pi_1=n-1$.

In summary, given $\sigma^{\prime} \in \mathcal{R}_{n-1}(2143)$, we can produce one member $\sigma=\pi\pi^r$ of $\mathcal{R}_{n}(2143)$ where $\pi_1=1$ and one where $\pi_1=n$.  There are 4 additional reverse double lists where $\pi_1=2$ or $\pi_1=n-1$, so for $n \geq 4$, we have 

$$\mathrm{r}_{n}(2143)=2\mathrm{r}_{n-1}(2143)+4.$$

From this recurrence it follows that $\mathrm{r}_{n}(2143) = 5\cdot 2^{n-2}-4$ for $n \geq 3$.
\end{proof}

\subsubsection{The pattern 1423}

\begin{theorem}
$\mathrm{r}_n(1423)=2\mathrm{r}_{n-1}(1423)+\mathrm{r}_{n-3}(1423)+2$ for $n\geq 5$.
\label{T:1423}
\end{theorem}

\begin{proof}
Suppose $n \geq 5$, let $\sigma=\pi\pi^r \in \mathcal{R}_n(1423)$, and let $\sigma^{\prime} \in \mathcal{R}_{n-1}(1423)$ be the reverse double list formed by deleting both copies of $n$ in $\sigma$.  We consider 5 cases based on the value of $\pi_1$.

Suppose $\pi_1=n$.  Since $n$ can only play the role of a $4$ in a $1423$ pattern, $n\sigma^{\prime}n \in \mathcal{R}_n(1423)$ for any $\sigma^{\prime} \in \mathcal{R}_{n-1}(1423)$.  There are $\mathrm{r}_{n-1}(1423)$ reverse double lists on $n$ letters in this case.

Suppose $\pi_1=n-1$.  We claim that $\pi_2=n$.  Assume for contradiction $\pi_2\neq n$.  This implies $\pi_2=1$ or $2\leq \pi_2\leq n-2$. If $\pi_2=1$, then the digits 1 and $n$ in $\pi$ along with the digits 2 and $n-1$ in $\pi^r$ form a 1423 pattern. If $2\leq \pi_2 \leq n-2$, then the digit 1 in $\pi$ along with the digits $n$, $\pi_2$, and $n-1$ in $\pi^r$ form a 1423 pattern. Thus, when $\pi_1=n-1$, $\pi_2=n$.  Notice that $\pi_2=n$ can only play the role of 4 in a 1423 pattern.  However, since $\pi_1=n-1$, it cannot  be part of a 1423 pattern in $\sigma$. On the other hand, $n-1$ can play the role of a 3 or a 4. If the copy of $n-1$ in $\pi^r$ serves as a 3 in an occurrence of 1423, then the copy of $n$ in $\pi$ must serve as the 4; however, there is no number to play the role of the 1 that precedes $n$. Therefore whenever $\sigma^{\prime\prime} \in \mathcal{R}_{n-2}(1423)$, we have $(n-1)n\sigma^{\prime\prime}n(n-1) \in \mathcal{R}_n(1423)$.  There are  $\mathrm{r}_{n-2}(1423)$ reverse double lists on $n$ letters in this case.

Suppose $3 \leq \pi_1 \leq n-2$.  We claim that $\pi_2=\pi_1+1$.  Let $\pi_1=a$.  Assume for contradiction $\pi_2\neq a+1$. This implies $\pi_2=1$, $2 \leq \pi_2\leq a-1$, $a+2 \leq \pi_2 \leq n-1$, or $\pi_2=n$.

If $\pi_2=1$, then the digits 1 and $n$ in $\pi$ along with the digits 2 and $a$ in $\pi^r$ form a 1423 pattern. If $2 \leq \pi_2 \leq a-1$, then the digit 1 in $\pi$ along with the digits $n$, $\pi_2$, and $a$ in $\pi^r$ form a 1423 pattern.  If $a+2 \leq \pi_2 \leq n-1$, then the digits $a$ and $n$ in $\pi$ along with the digits $a+1$ and $\pi_2$ in $\pi^r$ form a 1423 pattern. If $\pi_2=n$, then the digits  $a$, $n$, and $a+1$ in $\pi$ along with the digit $n-1$ in $\pi^r$  form a 1423 pattern.  Thus, when $\pi_1=a$, then $\pi_2=a+1$.  By a similar argument, $\pi_j=a+j-1$ for $1\leq j \leq n-a-1$.  

We have that $\pi_{n-a-1}=n-2$.  We claim that $\{\pi_{n-a},\pi_{n-a+1}\}=\{n-1,n\}$.  Suppose to the contrary that $\pi_{n-a} \notin \{n-1,n\}$.  If $\pi_{n-a}<a-1$, then the digits $\pi_{n-a}$ and $n$ in $\pi$ along with the digits $a-1$ and $a$ in $\pi^r$ form a 1423 pattern.  If $\pi_{n-a}=a-1$, then the digit 1 in $\pi$ along with the digits $n$, $a-1$, and $a$ in $\pi^r$ form a 1423 pattern.  So it must be the case that $\pi_{n-a} \in \{n-1,n\}$.  Similarly, $\pi_{n-a+1}=\{n-1,n\} \setminus \{\pi_{n-a}\}$.  Ultimately, either $\pi_{n-a}=n-1$ and $\pi_{n-a+1}=n$ or $\pi_{n-a}=n$ and $\pi_{n-a+1}=n-1$.

Now, the remaining $a-1$ positions of $\pi$ can be filled with $\sigma^* \in \mathcal{R}_{a-1}(1423)$.  Because $n$ and $n-1$ can be interchanged and because there exist $2\mathrm{r}_{a-1}(1423)$ reverse double lists on $n$ letters where $\pi_1=a$ for $3\leq a\leq n-2$, there are $\displaystyle \sum_{i=2}^{n-3} 2\mathrm{r}_i(1423)$ reverse double lists in this case.

Suppose that $\pi_1=2$.  We claim that either $\pi_2=1$ or $\pi_2=3$.  Assume for contradiction that $\pi_2 \notin \{1,3\}$. This implies $\pi_2=n$ or $4 \leq \pi_2\leq n-1$. If $\pi_2=n$, then the digits 2, $n$, and 3 in $\pi$ along with the digit $n-1$ in $\pi^r$  form a 1423 pattern. If $4 \leq \pi_2 \leq n-1$, then the digits 2 and $n$ in $\pi$ along with the digits $3$ and $\pi_2$ in $\pi^r$ form a 1423 pattern. Therefore, if $\pi_1=2$, then $\pi_2=1$ or $\pi_2=3$.  By a similar argument, $\pi_j \in \{1, \min(\{2,\dots, n\} \setminus \{\pi_1, \dots, \pi_{j-1}\})\}$ for $2 \leq j \leq n-2$.  Finally, $n-1$ may either precede or follow $n$.  We have $n-1$ choices for the location of 1 and 2 choices for the order of $n-1$ and $n$ in $\pi$, so there are $2(n-1)$ 1423-avoiding reverse double lists on $n$ letters where $\pi_1=2$.

Suppose that $\pi_1=1$.  We claim that $\pi_2=2$.  Assume for contradiction $\pi_2 \neq 2$.  Then the digits 1 and $n$ in $\pi$ along with the digits 2 and $\pi_2$ in $\pi^r$ form a 1423 pattern.  By a similar argument, $\pi_i=i$ for $1 \leq i \leq n-2$.  Finally, $\pi_{n-1}=n-1$ and $\pi_n=n$ or $\pi_{n-1}=n$ and $\pi_n=n-1$, giving two 1423-avoiding reverse double lists on $n$ letters when $\pi_1=1$.

Combining our 5 cases, we have shown that for $n \geq 5$,

\begin{equation}
\mathrm{r}_n(1423)=\mathrm{r}_{n-1}(1423)+\mathrm{r}_{n-2}(1423)+\sum_{i=2}^{n-3}2\mathrm{r}_i(1423)+2n.
\label{E:1423claim}
\end{equation}

We are ready to prove the $n \geq 5$ case of the theorem by induction.  For the base case, when $n=5$, $\mathrm{r}_5(1423)=36$, and $2\mathrm{r}_{4}(1423)+\mathrm{r}_{2}(1423)+2=2\cdot 16+2+2=36$, as desired.

Now, suppose $\mathrm{r}_k(1423)=2\mathrm{r}_{k-1}(1423)+\mathrm{r}_{k-3}(1423)+2$ for some $k\geq 5$. From Equation \ref{E:1423claim}, we have:

\begin{equation}
\begin{split}
 \mathrm{r}_{k+1}(1423) 
&=\mathrm{r}_k(1423)+\mathrm{r}_{k-1}(1423)+2(k+1)+\sum_{i=2}^{k-2}2\mathrm{r}_i(1423)\\
&=\mathrm{r}_k(1423)+\mathrm{r}_{k-1}(1423)+2k+2+\sum_{i=2}^{k-3}2\mathrm{r}_i(1423)+2\mathrm{r}_{k-2}(1423) \\
&=\mathrm{r}_k(1423)+2\mathrm{r}_{k-2}(1423)+\mathrm{r}_{k-1}(1423)+2k+2+\sum_{i=2}^{k-3}2\mathrm{r}_i(1423).
\end{split} 
 \label{E:1423induction}
 \end{equation}

On the other hand, when $n=k$, Equation \ref{E:1423claim} gives:
\begin{equation} 
 \sum_{i=2}^{k-3}2\mathrm{r}_i(1423)=\mathrm{r}_k(1423)-\mathrm{r}_{k-1}(1423)-\mathrm{r}_{k-2}(1423)-2k.
 \label{E:1423algebra}
\end{equation}

After substituting Equation \ref{E:1423algebra} into Equation \ref{E:1423induction}, we have:
\begin{equation*}
\begin{split}
\mathrm{r}_{k+1}(1423)=\mathrm{r}_k(1423)
&+ 2\mathrm{r}_{k-2}(1423)+\mathrm{r}_{k-1}(1423)+2k+2\\
&+ \mathrm{r}_k(1423)-\mathrm{r}_{k-1}(1423)-\mathrm{r}_{k-2}(1423)-2k,
\end{split}
\end{equation*}
which simplifies to 

\begin{equation*}
\begin{split}
 \mathrm{r}_{k+1}(1423)
 &=2\mathrm{r}_{k}(1423)+\mathrm{r}_{k-2}(1423)+2\\
 &= 2\mathrm{r}_{(k+1)-1}(1423)+\mathrm{r}_{(k+1)-3}(1423)+2.
\end{split}
\end{equation*}

\end{proof}

\subsubsection{The pattern 1432}

\begin{theorem}
$\mathrm{r}_n(1432)=2\mathrm{r}_{n-1}(1432)+\mathrm{r}_{n-2}(1432)$ for $n\geq 5$.
\label{T:1432}
\end{theorem}

\begin{proof}
Suppose $n \geq 5$, and let $\sigma=\pi\pi^r \in \mathcal{R}_n(1432)$.  We consider 5 cases based on the value of $\pi_1$.

Suppose $\pi_1=1$.  Then $2 \leq \pi_2 \leq n-2$ or $\pi_2\in\{n-1,n\}$.  If $2 \leq \pi_2 \leq n-2$, then the digits 1 and $n$ in $\pi$ along with the digits $n-1$ and $\pi_2$ in $\pi^r$ form a 1432 pattern.  If $\pi_2=n-1$ or $\pi_2=n$, the the digits 1, $\pi_2$, and $n-2$ in $\pi$ along with the digit $n-3$ in $\pi^r$ form a 1432 pattern. Thus, $\pi_1\neq 1$.
	
Suppose $2 \leq \pi_1 \leq n-3$.  Then $\pi_2=n$ or $1\leq\pi_2\leq n-1$. If $\pi_2=n$, then the digits $\pi_1$, $n$, and $n-1$ in $\pi$ along with the digit $n-2$ in $\pi^r$ form a 1432 pattern. If $1\leq \pi_2\leq n-1$, then two cases must be considered. If $\pi_1<\pi_2$, then the digit 1 in $\pi$ along with the digits $n$, $\pi_2$, and $\pi_1$ in $\pi^r$ form a 1432 pattern. If $\pi_1>\pi_2$, then the digits $\pi_2$ and $n$ in $\pi$ along with the digits $n-1$ and $\pi_1$ in $\pi^r$ form a 1432 pattern. Thus, $\pi_1 \geq n-2$.

Suppose $\pi_1=n-2$.  We claim that $\pi_2=n$.  Assume for contradiction that $\pi_2 \neq n$.  This implies $\pi_2=n-1$ or $1 \leq \pi_2 \leq n-3$. If $\pi_2=n-1$, then the digit 1 in $\pi$ along with the digits $n$, $n-1$, and $n-2$ in $\pi^r$ form a 1432 pattern. If $1 \leq \pi_2 \leq n-3$, then the digits $\pi_2$ and $n$ in $\pi$ along with the digits $n-1$ and $n-2$ in $\pi^r$ form a 1432 pattern.  Thus, if $\pi_1=n-2$, then $\pi_2=n$. 
	
Now, $n$ can only play the role of a 4 in a 1432 pattern in $\sigma$, but $\pi_1=n-2$ prevents the first copy of $n$ from being in such a pattern, and the fact that it is the penultimate digit of $\pi^r$ prevents the second copy of $n$ from being in such a pattern. Also, $n-2$ can only play the role of 2, 3, or 4, so it will not play the role of 1 in the beginning of a $1432$ pattern. If the second copy of $n-2$ serves as a 2, then $n-1$ in $\pi^r$ can serve as the 3 and $n$ in $\pi$ must serve as the $4$. However, there is no number to play the role of 1 that precedes $n$ in $\pi$. Thus, the remaining positions can be filled in with any member of $\mathcal{R}_{n-2}(1432)$.  There are $\mathrm{r}_{n-2}(1432)$ 1432-avoiding reverse double lists on $n$ letters where $\pi_1=n-2$.  

Suppose $\pi_1=n-1$.  Since $n-1$ can only play the role of a 3 or a 4 in a $1432$ pattern, it cannot play the role of 1 at the beginning or 2 at the end of such a pattern.  There are $\mathrm{r}_{n-1}(1432)$ ways to fill in the remaining digits of $\sigma$, so there are $\mathrm{r}_{n-1}(1432)$ 1432-avoiding reverse double lists on $n$ letters where $\pi_1=n-1$. 

Suppose $\pi_1=n$.  The only role $n$ can play in a $1432$ pattern is a 4, however $n$ only appears as the first and the last member of $\sigma$, so it cannot be involved in a 1432 pattern.  There are $\mathrm{r}_{n-1}(1432)$ ways to fill in the remaining digits of $\sigma$, so there are $\mathrm{r}_{n-1}(1432)$ 1432-avoiding reverse double lists on $n$ letters where $\pi_1=n$.

Combining our 5 cases, we have that for $n \geq 5$,

$$\mathrm{r}_n(1432)=2\mathrm{r}_{n-1}(1432)+\mathrm{r}_{n-2}(1432).$$
\end{proof}

\subsubsection{The pattern 1342}

\begin{theorem}
$\mathrm{r}_n(1342)= 2\mathrm{r}_{n-1}(1342)+\mathrm{r}_{n-2}(1342)+2$ for $n\geq 4$.
\label{T:1342}
\end{theorem}

\begin{proof}
Suppose $n \geq 4$, and let $\sigma=\pi\pi^r \in \mathcal{R}_n(1342)$.  We consider 5 cases based on the value of $\pi_1$.

Suppose $\pi_1=1$. Then $\pi_2=n$.  Assume for contradiction that $\pi_2\neq n$. This implies that $\pi_2=2$ or $3\leq\pi_2\leq n-1$. If $\pi_2=2$, then the digits 1 and 3 in $\pi$ along with the digits 4 and 2 in $\pi^r$  form a 1342 pattern. If $3\leq\pi_2\leq n-1$, then the digits 1, $\pi_2$, and $n$ in $\pi$ along with the digit 2 in $\pi^r$ form a 1342 pattern. Hence, when $\pi_1=1$, $\pi_2=n$. 
By a similar argument, $\pi_i=n-i+2$ for $2 \leq i \leq n-2$.  Finally, either $\pi_{n-1}=2$ and $\pi_n=3$ or $\pi_{n-1}=3$ and $\pi_n=2$. Thus, there are two ways to avoid $1342$ when $\pi_1=1$.

Suppose $2\leq\pi_1\leq n-3$. Then $\pi_2=1$, $2\leq\pi_2\leq n-2$, or $\pi_2\in\{n-1,n\}$. If $\pi_2=1$, then the digits 1 and $n-1$ in $\pi$ along with the digits $n$ and $\pi_1$ in $\pi^r$ form a 1342 pattern. If $2\leq\pi_2\leq n-2$, then two cases must be considered. If $\pi_1<\pi_2$, then the digits $\pi_1$ and $n-1$ in $\pi$ along with the digits $n$ and $\pi_2$ in $\pi^r$ form a 1342 pattern. If $\pi_1>\pi_2$, then the digits $\pi_2$ and $n-1$ in $\pi$ along with the digits $n$ and $\pi_1$ in $\pi^r$ form a 1342 pattern. If $\pi_2\in\{n-1,n\}$, then the digit 1 in $\pi$ along with the digits $n-2$, $\pi_2$, and $\pi_1$ in $\pi^r$ form a 1342 pattern. Hence, there are no 1342-avoiding reverse double lists where $2\leq\pi_1\leq n-3$.

Suppose $\pi_1=n-2$.  Then $\pi_2=n-1$.  Assume for contradiction that $\pi_2\neq n-1$. This implies that $\pi_2=n$ or $1\leq\pi_2\leq n-3$. If $\pi_2=n$, then the digit 1 in $\pi$ along with the digits $n-1$, $n$, and $n-2$ in $\pi^r$ form a 1342 pattern. If $1\leq\pi_2\leq n-3$, then the digits $\pi_2$ and $n-1$ in $\pi$ along with the digits $n$ and $n-2$ in $\pi^r$ form a 1342 pattern. Thus, if $\pi_1=n-2$, then $\pi_2=n-1$.  Notice that $n-1$ can only play the role of a 3 or 4 in a 1342 pattern. Also, $n-2$ cannot play the role of a $1$. If $n-2$ plays the role of 2, then $n$ in $\pi^r$ must play the role of 4. Now, the only number to play the role of 3 is $n-1$ in $\pi$. However, there is no digit that can play the role of a 1 that precedes $n-1$ in $\pi$. Thus, the remaining positions can be filled in $\mathrm{r}_{n-2}(1342)$ ways.

Suppose $\pi_1=n-1$.  Since $n-1$ can only play the role of a 3 or a 4 in a $1342$ pattern, it cannot play the role of 1 at the beginning or 2 at the end of such a pattern.  There are $\mathrm{r}_{n-1}(1342)$ ways to fill in the remaining digits of $\sigma$, so there are $\mathrm{r}_{n-1}(1342)$ 1342-avoiding reverse double lists on $n$ letters where $\pi_1=n-1$. 

Suppose $\pi_1=n$.  The only role $n$ can play in a $1342$ pattern is a 4, however $n$ only appears as the first and the last member of $\sigma$, so it cannot be involved in a 1342 pattern.  There are $\mathrm{r}_{n-1}(1342)$ ways to fill in the remaining digits of $\sigma$, so there are $\mathrm{r}_{n-1}(1342)$ 1342-avoiding reverse double lists on $n$ letters where $\pi_1=n$.

Combining all 5 cases, for $n \geq 4$, we have shown:

$$\mathrm{r}_n(1342)=	2\mathrm{r}_{n-1}(1342) + \mathrm{r}_{n-2}(1342) + 2.$$

\end{proof}

\subsubsection{The pattern 2413}

\begin{theorem}
$\mathrm{r}_n(2413)= 2\mathrm{r}_{n-1}(2413)+2\mathrm{r}_{n-2}(2413)$ for $n\geq 3$.
\label{T:2413}
\end{theorem}

\begin{proof}
The cases where $n \leq 4$ are easy to check by brute force, so we focus on $n \geq 5$.  Let $\sigma=\pi\pi^r \in \mathcal{R}_n(2413)$, and let $\sigma^{\prime} \in \mathcal{R}_{n-1}(2413)$ be the reverse double list formed by deleting both copies of $n$ in $\sigma$.  We consider 5 cases based on the value of $\pi_1$.

Suppose $\pi_1=1$. Since 1 can only play the role of a 1 in a 2413 pattern, adding $1$ to the beginning and end of $\sigma^{\prime} \in \mathcal{R}_{n-1}(2413)$ will not create a 2413 pattern. Thus, there are $\mathrm{r}_{n-1}(2413)$ ways to create a 2413-avoiding reverse double list on $n$ letters where $\pi_1=1$.

Suppose $\pi_1=2$. We claim that $\pi_2=1$.  Assume for contradiction that $\pi_2\neq 1$. This implies that $\pi_2=n$ or $3\leq\pi_2\leq n-1$. If $\pi_2=n$, then the digits 2, $n$, and 1 in $\pi$ along with the digit 3 in $\pi^r$ form a 2413 pattern. If $3\leq\pi_2\leq n-1$, then the digits 2 and $n$ in $\pi$ along with the digits 1 and $\pi_2$ in $\pi^r$ form a 2413 pattern. Hence, if $\pi_1=2$, then $\pi_2=1$. 
Notice that 2 can only play the role of 1 or 2 in a 2413 pattern, but its location forces 2 to either be the first digit or last digit in a 2413 pattern.  If 2 is involved in a 2413 pattern, it plays the role of 2. If it plays the role of 2 then the 1 in $\pi^r$ must play the role of 1, but there is no digit after 1 that can play the role of 3. Also, 1 can only play the role of 1, but the location of 1 prevents either copy from participating in a 2413 pattern. Now, the remaining positions can be filled in $\mathrm{r}_{n-2}(2413)$ ways to avoid a $2413$ pattern.

Suppose $3 \leq \pi_1 \leq n-2$.  Now, $\pi_2=n$, $2\leq \pi_2 \leq n-1$, or $\pi_2=1$.  If $\pi_2=n$, the digits $\pi_1$, $n$, and 1 in $\pi$ along with the digit $n-1$ in $\pi^r$ form a 2413 pattern.  If $2\leq\pi_2\leq n-1$, two cases need to be considered. If $\pi_1>\pi_2$, then the digits $\pi_2$ and $n$ in $\pi$ along with the digits 1 and $\pi_1$ in $\pi^r$ form a 2413 pattern. If $\pi_1<\pi_2$, then the digits $\pi_1$ and $n$ in $\pi$ along with the digits 1 and $\pi_2$ in $\pi^r$ form a 2413 pattern.  If $\pi_2=1$, the digit 2 in $\pi$ along with the digits $n$, $1$, and $\pi_1$ in $\pi^r$ form a 2413 pattern.  In every case, there are no 2413-avoiding reverse double lists on $n$ letters where $3 \leq \pi_1 \leq n-2$. 

Suppose $\pi_1=n-1$. We claim that $\pi_2=n$. Assume for contradiction that $\pi_2\neq n$. This implies $\pi_2=1$ or $2\leq\pi_2\leq n-2$. If $\pi_2=1$, then the digit 2 in $\pi$ along with the digits $n$, 1, and $n-1$ in $\pi^r$ form a 2413 pattern. If $2\leq\pi_2\leq n-2$, then the digits $\pi_2$ and $n$ in $\pi$ along with the digits $1$ and $n-1$ in $\pi^r$ form a 2413 pattern. Hence, if $\pi_1=n-1$, then $\pi_2=n$. 
Notice that $n-1$ can only play the role of a 3 or 4 in a 2413 pattern, so it cannot play as the 2 at the beginning of a $2413$ pattern. If $n-1$ plays the role of 3, the number $n$ in $\pi$ can play the role of 4. However, there does not exist a number before $n$ in $\pi$ to play the role of 2. Also, $n$ can only play the role of a 4, but the location of $n$ prevents either copy from participating in a 2413 pattern. Now, the remaining positions can be filled in $\mathrm{r}_{n-2}(2413)$ ways to avoid a $2413$ pattern.

Suppose $\pi_1=n$.  The only role $n$ can play in a $2413$ pattern is a 4, however $n$ only appears as the first and the last member of $\sigma$, so it cannot be involved in a 2413 pattern.  There are $\mathrm{r}_{n-1}(2413)$ ways to fill in the remaining digits of $\sigma$, so there are $\mathrm{r}_{n-1}(2413)$ 2413-avoiding reverse double lists on $n$ letters where $\pi_1=n$.

Combining our cases, we have shown that for $n \geq 3$,

$$\mathrm{r}_n(2413)=2\mathrm{r}_{n-1}(2413)+2\mathrm{r}_{n-2}(2413).$$

\end{proof}

\subsubsection{Summary of length 4 patterns}

We have now completely characterized $\mathrm{r}_n(\rho)$ where $\rho$ is a permutation pattern of length at most 4. By exploiting the symmetry inherent in reverse double lists, we found recurrences for $\mathrm{r}_n(\rho)$ for each pattern of length 4.  The corresponding results are given in Table \ref{T:length4formula}.  These results provide an interesting contrast to pattern-avoiding permutations and double lists.  First, there is exactly one non-trivial Wilf equivalence.  Second, the monotone pattern is the hardest pattern to avoid in the context of reverse double lists.  Finally, we obtained a variety of behaviors (constant, cubic, and exponential), as compared to permutation-pattern sequences which only grow exponentially.

\begin{table}[hbt]
\begin{center}
\begin{tabular}{|l|ll|}
\hline
Pattern $\rho$& $\mathrm{r}_n(\rho)$&\\
\hline
$1234 \sim 4321$ & 0 & ($n\geq 7$) \\
\hline
$1243  \sim 2134 \sim 3421  \sim 4312$ & $\frac{n^3}{3}-\frac{7n}{3}+4$ & ($n\geq 2$)\\
\hline
$1324 \sim 4231 \sim 2143 \sim 3412$&$5\cdot 2^{n-2}-4$ & ($n\geq 4$)\\
\hline
$1423 \sim 2314 \sim 3241 \sim 4132$ & $2\mathrm{r}_{n-1}(\rho)+\mathrm{r}_{n-3}(\rho)+2$ & ($n\geq 5$) \\
\hline
$1432  \sim 2341  \sim 3214  \sim 4123 $ &$2\mathrm{r}_{n-1}(\rho)+\mathrm{r}_{n-2}(\rho)$ &($n\geq 5$)\\
\hline
$1342 \sim 2431 \sim 3124 \sim 4213$ & $2\mathrm{r}_{n-1}(\rho)+\mathrm{r}_{n-2}(\rho)+2$ & ($n\geq 4$)\\
\hline
$2413 \sim 3142$&$2\mathrm{r}_{n-1}(\rho)+2\mathrm{r}_{n-2}(\rho)$& ($n\geq 3$)\\
\hline
\end{tabular}
\end{center}
\caption{Enumeration of reverse double lists avoiding a pattern of length 4}
\label{T:length4formula}
\end{table}

Each formula in Table \ref{T:length4formula} is straightforward to convert to a generating function via standard techniques.  The corresponding generating functions $\displaystyle{\sum_{n=0}^\infty \mathrm{r}_n(\rho) x^n}$ are given in Table \ref{T:length4gf}.  Because each generating function is rational, we can find the corresponding linear recurrence satisfied by each sequence $\left\{\mathrm{r}_n(\rho)\right\}_{n \geq 0}$ and determine the largest root of the corresponding characteristic equation to determine the exponential growth rate of the sequence.  The table includes the exact growth rate for every sequence except  $\left\{\mathrm{r}_n(1423)\right\}_{n \geq 0}$.  In that case, the largest root of the characteristic equation is $$\frac{2}{3} + \frac{1}{3}\sqrt[3]{\frac{1}{2}(43-3\sqrt{177})} +\frac{1}{3}\sqrt[3]{\frac{1}{2}(43+3\sqrt{177})} \approx 2.21.$$

\begin{table}[hbt]
\begin{center}
\begin{tabular}{|l|c|l|}
\hline
Pattern $\rho$& ogf for $\mathrm{r}_n(\rho)$&exponential growth rate\\
\hline
\hline
\multirow{2}{*}{1234}&\scalebox{0.75}{$1+x+2x^2+6x^3$}&\multirow{2}{*}{}\\
&\scalebox{0.75}{\phantom{XXXXX}$+16x^4+32x^5+32x^6$}&\\
\hline
1243&$\frac{-x^5+x^4+4x^2-3x+1}{(x-1)^4}$ & \\
\hline
1324&\multirow{2}{*}{$\frac{2x^4+2x^3+x^2-2x+1}{(x-1)(2x-1)}$} & \multirow{2}{*}{$2$}\\
2143&&\\
\hline
1423&$\frac{-x^5+2x^4+x^3+x^2-2x+1}{(x-1)(x^3+2x-1)}$ & $\approx 2.21$\\
\hline
1432&$\frac{-2x^4-x^3+x^2+x-1}{x^2+2x-1}$ & $(1+\sqrt{2}) \approx 2.41$\\
\hline
1342&$\frac{x^4+2x^3-2x+1}{(x-1)(x^2+2x-1)}$ & $(1+\sqrt{2}) \approx 2.41$\\
\hline
2413&$\frac{(x+1)(2x-1)}{2x^2+2x-1}$& $(1+\sqrt{3}) \approx 2.73$\\
\hline
\end{tabular}
\end{center}
\caption{Generating functions and exponential growth rates for avoiding a length 4 pattern}
\label{T:length4gf}
\end{table}

\section{Avoiding a pattern of length 5 or more}\label{S:five}

There are 32 trivial Wilf classes for patterns of length 5.  Table \ref{T:length5data} shows the brute force data for $\left\{\mathrm{r}_n(\rho)\right\}_{n=5}^{7}$ for one pattern $\rho$ from each trivial Wilf class.  From this data it is clear that there are no non-trivial Wilf equivalences for patterns of length 5.  There are some interesting observations that arise from the data.  Notice that $\mathrm{r}_7(15243)=\mathrm{r}_7(15324)$ while $\mathrm{r}_6(15243)\neq \mathrm{r}_6(15324)$. From brute force data, it appears that while $\mathrm{r}_6(15243)<\mathrm{r}_6(15324)$, for $n>7$ $\mathrm{r}_n(15243)>\mathrm{r}_n(15324)$.   Similarly $\mathrm{r}_6(23514)< \mathrm{r}_6(13425)$ and $\mathrm{r}_7(23514)=\mathrm{r}_7(13425)$, while it appears that $\mathrm{r}_n(23514)> \mathrm{r}_n(13425)$  for $n>7$. This behavior is in contrast to enumeration sequences for pattern-avoiding permutations and double lists where there is no known example where $\mathrm{s}_N(\rho)<\mathrm{s}_N(\tau)$ (resp. $\mathrm{d}_N(\rho)<\mathrm{d}_N(\tau)$) for some integer $N$ but $\mathrm{s}_n(\rho)\geq \mathrm{s}_n(\tau)$ (resp. $\mathrm{d}_n(\rho)\geq \mathrm{d}_n(\tau)$) for $n \geq N$.

\begin{table}
\begin{center}
\begin{tabular}{|l|l||l|l|}
\hline
Pattern $\rho$& $\left\{\mathrm{r}_n(\rho)\right\}_{n=5}^{7}$&Pattern $\rho$& $\left\{\mathrm{r}_n(\rho)\right\}_{n=5}^{7}$\\
\hline
\hline
$12345^\dagger$&104, 432, 1584&$15324^\dagger$&104, 442, 1772\\
\hline
$12354^\dagger$&104, 434, 1630&$21543^\dagger$&104, 442, 1800\\
\hline
$15423^\dagger$&104, 434, 1706&$13425^\dagger$&104, 444, 1808\\
\hline
$21354^\dagger$&104, 436, 1676&$14325^\dagger$&104, 444, 1828\\
\hline
21534&104, 436, 1746&$25314^\dagger$&104, 444, 1868\\
\hline
$14523^\dagger$&104, 436, 1748&15342&104, 444, 1880\\
\hline
$12534^\dagger$&104, 438, 1710&$25413^\dagger$&104, 444, 1884\\
\hline
13254&104, 438, 1720&$24513^\dagger$&104, 444, 1888\\
\hline
$12435^\dagger$&104, 438, 1726&$15432^\dagger$&104, 446, 1846\\
\hline
12543&104, 438, 1766&14253&104, 448, 1904\\
\hline
$15234^\dagger$&104, 440, 1704&14532&104, 448, 1914\\
\hline
$12453^\dagger$&104, 440, 1750&14352&104, 448, 1924\\
\hline
$21453^\dagger$&104, 440, 1766&$13524^\dagger$&104, 450, 1926\\
\hline
15243&104, 440, 1772&13542&104, 452, 1958\\
\hline
$13452^\dagger$&104, 440, 1802&25143&104, 454, 1982\\
\hline
$23514^\dagger$&104, 440, 1808&24153&104, 454, 1990\\
\hline
\end{tabular}
\end{center}
\caption{Enumeration of reverse double lists avoiding a pattern of length 5}
\label{T:length5data}
\end{table}

We can say more about the growth rates of these sequences by relating $\mathrm{r}_n(\rho)$ to pattern-avoiding permutations, as described in Theorem \ref{T:rdl2perm}.  First, recall that a shuffle of words $\alpha_1\cdots \alpha_i$ and $\beta_1 \cdots \beta_j$ is a word $w$ of length $i+j$ where there is a subsequence of $w$ equal to $\alpha$, and a disjoint subsequence of $w$ equal to $\beta$.  Now, given a permutation $\rho \in \mathcal{S}_k$ define $\rho^{\leftrightarrow}$ to be the set of permutations that are shuffles of $\rho_1 \cdots \rho_i$ and $\rho_k \cdots \rho_{i+1}$ for any $1 \leq i \leq k$.  For example, $1234^{\leftrightarrow}=\{1234, 1243, 1423, 4123, 1432, 4132, 4312, 4321\}$.  In general, there are $\binom{k-1}{i-1}$ ways to shuffle $\rho_1 \cdots \rho_i$ with $\rho_k \cdots \rho_{i+1}$ where $\rho_1 \cdots \rho_{i+1}$ is not a subsequence of the resulting word.  Summing over all possible values of $i$, if $\rho \in \mathcal{S}_k$, then $\displaystyle{\left|\rho^{\leftrightarrow}\right| = \sum_{i=1}^k\binom{k-1}{i-1} = 2^{k-1}}$.

\begin{theorem}\label{T:rdl2perm}
Given $\rho \in \mathcal{S}_k$ and $n \geq 0$, 

$$\mathrm{r}_n(\rho) = \mathrm{s}_n(\rho^{\leftrightarrow}).$$
\end{theorem}

\begin{proof}
Suppose $\sigma =\pi\pi^r \in \mathcal{R}_n$ contains $\rho$.  Then, for some $1 \leq i \leq n$, $\rho_1\cdots \rho_i$ is contained in $\pi$ while $\rho_{i+1} \cdots \rho_k$ is contained in $\pi^r$.  If $\rho_{i+1} \cdots \rho_k$ is contained in $\pi^r$, then $\rho_k \cdots \rho_{i+1}$ is contained in $\pi$.  Further,  $\rho_1\cdots \rho_i$ and $\rho_k \cdots \rho_{i+1}$ use disjoint digits of $\sigma$, so $\pi$ contains a shuffle of $\rho_1\cdots \rho_i$ and $\rho_k \cdots \rho_{i+1}$.  It follows that $\sigma$ avoids $\rho$ if and only if $\pi$ avoids all shuffles of $\rho_1\cdots \rho_i$ and $\rho_k \cdots \rho_{i+1}$ for all $1 \leq i \leq k$.
\end{proof}

As a consequence of Theorem \ref{T:rdl2perm}, when $\rho \in \mathcal{S}_k$, $\mathcal{R}_n(\rho)$ is isomorphic to a classical permutation class $\mathcal{S}_n(B)$, where $B$ is a set of $2^{k-1}$ classical permutations of length $k$.  For example, $\mathrm{r}_n(123) = \mathrm{s}_n(123, 132, 312, 321)$.  Because every set we have enumerated in this paper is a finitely-based classical permutation class, we can use known machinery to determine asymptotic growth of $\mathrm{r}_n(\rho)$ for arbitrary $\rho$.  For example, Vatter showed the following: 

\begin{theorem}[\cite{Vatter06}, Theorem 7]
Let $B$ be a finite set of patterns. The pattern-avoidance tree $T(B)$ is isomorphic
to a finitely labeled generating tree if and only if $B$ contains both a child of an increasing
permutation and a child of a decreasing permutation.
\label{T:finlabel}
\end{theorem}

Theorem \ref{T:finlabel} implies that $\mathrm{s}_n(B)$ has a rational generating function if $B$ contains both a permutation that is achieved by inserting one digit into an increasing permutation and another permutation that is achieved by inserting one digit into a decreasing permutation.  The patterns $\rho$ for which $\rho^{\leftrightarrow}$ fits this criteria are marked with $\dagger$ in Table \ref{T:length5data}.

Rational generating functions can indicate either polynomial or exponential growth, however.  In \cite{AAB07}, Albert, Atkinson, and Brignall give necessary and sufficient conditions on when $\mathrm{s}_n(B)$ exhibits polynomial growth.  The direct sum $\alpha \oplus \beta$ of permutations $\alpha$ and $\beta$ is the permutation formed by concatenating $\alpha$ and $\beta$ and incrementing all digits of $\beta$ by $\left|\alpha\right|$.  The skew sum $\alpha \ominus \beta$ of permutations $\alpha$ and $\beta$ is the permutation formed by concatenating $\alpha$ and $\beta$ and incrementing all digits of $\alpha$ by $\left|\beta\right|$.  Further let $\epsilon=(e_1,e_2)$ be an ordered pair where $\{e_1,e_2\} \subseteq \{-1,1\}$.  The pattern class $W(\epsilon)$ is the set of all permutations $\pi=\pi_1\cdots \pi_n$ where there exists a value $1 \leq j \leq n$ such that $\pi^{(i)}$ is increasing if $e_i=1$ and $\pi^{(i)}$ is decreasing if $e_i=-1$ where $\pi^{(1)}:=\pi_1\cdots \pi_j$ and $\pi^{(2)}:=\pi_{j+1}\cdots \pi_n$.  We have the following characterization from Albert, Atkinson, and Brignall:

\begin{theorem}[\cite{AAB07}, Theorem 1]
$\mathrm{s}_n(B)$ has polynomial growth if and only if $B$ contains a member of each of the following 10 sets of permutations: $W(1,1)$, $W(1,-1)$, $W(-1,1)$, $W(-1,-1)$, $W(1,1)^{-1}$, $W(1,-1)^{-1}$, $W(-1,1)^{-1}$, \\$W(-1,-1)^{-1}$, $L_2 = \left\{\bigoplus_{i=1}^j \alpha_i \middle| \alpha_i \in \{1, 21\} \text{ for all }i\right\}$ and $L_2^r$.
\label{T:polychar}
\end{theorem}

Applying Theorem \ref{T:polychar} to the case of reverse double lists yields the following:

\begin{theorem}
$\mathrm{r}_n(\rho)$ has polynomial growth if and only if $\rho$ is trivially Wilf equivalent to $12\cdots k$ or to $1\cdots (k-2)k(k-1)$.
\label{T:poly}
\end{theorem}

\begin{proof}

We can check that the theorem holds for $\rho$ of length at most 9 by brute force methods, so without loss of generality, we assume that $k \geq 10$.

First, notice that $\mathrm{s}_n(\rho^{\leftrightarrow})$ meets the criteria in Theorem \ref{T:polychar} when $\rho=12\cdots k$ or $\rho=1\cdots (k-1)k(k-1)$.  In the case of $\rho=1\cdots k$, $\rho$ is a member of classes $W(1,1)$, $W(1,-1)$, $W(-1,1)$, $W(1,1)^{-1}$, $W(1,-1)^{-1}$, $W(-1,1)^{-1}$, and $L_2$, while $\rho^r$ is a member of the other 3 classes, and $\{\rho, \rho^r\} \subset \rho^{\leftrightarrow}$.  In the case of $\rho=1\cdots (k-2)k(k-1)$, $1\cdots k$ is still a member of 7 of the 10 classes, $\rho^r$ is a member of the other 3 classes, and $\{1\cdots k,\rho^r\} \subset \rho^{\leftrightarrow}$.

Now, suppose that $\rho^{\leftrightarrow}$ has a member of each of the 10 necessary permutation classes.  This means there is some integer $i$ such that a shuffle of $\rho_1 \cdots \rho_i$ together with $\rho_n \cdots \rho_{i+1}$ is in $L_2$.  By definition of shuffle, the first digit of the shuffle is either $\rho_1$ or $\rho_n$.  By definition of $L_2$, the first digit of the shuffle is either 1 or 2.

Similarly, there is some integer $j$ such that a shuffle of $\rho_1 \cdots \rho_j$ together with $\rho_n \cdots \rho_{j+1}$ is in $L_2^r$.  By definition of shuffle, the first digit of the shuffle is either $\rho_1$ or $\rho_n$.  By definition of $L_2^r$, the first digit of the shuffle is either $n-1$ or $n$.

Combining these observations, either $\rho_1 \in \{1,2\}$ and $\rho_n \in \{n-1,n\}$ or $\rho_1 \in \{n-1,n\}$ and $\rho_n \in \{1,2\}$.  We will assume that $\rho_1 \in \{1,2\}$ because if $\rho_1 \in \{n-1,n\}$, then $\rho_1^c \in \{1,2\}$ and $\rho^c$ is Wilf-equivalent to $\rho$.

Now, since $\rho_n \in \{n-1,n\}$ and there is a shuffle of $\rho_1 \cdots \rho_i$ together with $\rho_n \cdots \rho_{i+1}$ in $L_2$, it must be the case that $\rho_n$ is among the last 3 digits of the shuffle.  In other words $i \geq n-3$.  This implies that $\rho_1 \cdots \rho_{n-3} \in L_2$.

Similarly, since $\rho_1 \in \{1,2\}$ and there is a shuffle of $\rho_1 \cdots \rho_j$ together with $\rho_n \cdots \rho_{j+1}$ in $L_2^r$, it must be the case that $\rho_1$ is among the last 3 digits of the shuffle.  In other words, $i \leq 3$.  This implies that $\rho_n \cdots \rho_4 \in L_2^r$, and by taking reversal, $\rho_4 \cdots \rho_n \in L_2$.

We assume that $n \geq 10$, so $\rho_1 \cdots \rho_{n-3}$ and $\rho_4 \cdots \rho_n$ overlap by at least 4 digits.  It must then be the case that $\rho \in L_2$ if $\rho^{\leftrightarrow}$ has nontrivial intersection with both $L_2$ and $L_2^r$.

Now, we know that $\rho \in \left\{ \bigoplus_{i=1}^j \alpha_i \middle| \alpha_i \in \{1, 21\}\right\}$.  Assume $\alpha_i=21$ for some $1<i<j$.  In other words $\rho$ has a layer of size 2 that is not at the beginning or the end of $\rho$.  We claim that there is no member of $\rho^{\leftrightarrow}$ in $W(-1,1)^{-1}$.  Suppose to the contrary that there is such a member of $\rho^{\leftrightarrow}$.  Write $\rho=\rho_1\cdots \rho_\ell \alpha_1\alpha_2 \rho_{\ell+3}\cdots \rho_n$. where $\alpha_1>\alpha_2$ are the members of $\alpha_i$. Clearly $\rho \notin W(-1,1)^{-1}$, so there must be an integer $m$ such that a shuffle of $\rho_1 \cdots \rho_m$ and $\rho_{n} \cdots \rho_{m+1}$ is in $W(-1,1)^{-1}$.  If $m\leq\ell$ then in the shuffle $\rho_{\ell+3}$ precedes $\alpha_2$ which precedes $\alpha_1$, and these three digits for a 312 pattern.  But all members of $W(-1,1)^{-1}$ avoid 312.  If $m \geq \ell+2$ then $\rho_\ell$ precedes $\alpha_1$ which precedes $\alpha_2$, and these three digits form a 132 pattern.  But all members of $W(-1,1)^{-1}$ avoid 132.  So it must be the case that $m=\ell+1$.  Thus $\rho_{\ell}$ precedes $\alpha_1$ and $\rho_{\ell+3}$ precedes $\alpha_2$.  Since $\rho_{\ell}<\alpha_1$, it must be the case that all digits larger than $\alpha_1$ appear after $\alpha_1$ in increasing order.  In other words $\rho_{\ell+3}$ must appear after $\alpha_1$.  But then $\rho_{\ell}$, $\alpha_1$, $\rho_{\ell+3}$ and $\alpha_2$ form a 1342 pattern, and all members of $W(-1,1)^{-1}$ avoid 1342.  We have reached a contradiction in every possible scenario, so the $\alpha_i=21$ is only possible if $i=1$ or $i=j$.

Now, suppose that $\rho = 21 \oplus \left(\bigoplus_{i=1}^{n-4} 1 \right)\oplus 21$.  We claim that there is no member of $\rho^{\leftrightarrow}$ in $W(1,-1)$. Suppose there is an integer $j$ such that a shuffle of $\rho_1\cdots \rho_j$ and $\rho_n \cdots \rho_{j+1}$ is in $W(1,-1)$. Since $\rho_1>\rho_2$, it must be the case that $j=1$, and the shuffle in question is $\rho^r = 12 \ominus \left(\bigominus_{i=1}^{n-4} 1 \right)\ominus 12$.  However, the longest increasing sequence at the beginning of $\rho^r$ is $\rho_n\rho_{n-1}$, and then $\rho_{n-2}\cdots \rho_1$ is not in decreasing order because $\rho_2<\rho_1$.  So, there is no member of $\rho^{\leftrightarrow}$ in $W(1,-1)$.

We have now shown that if $\rho^{\leftrightarrow}$ has a nontrivial intersection with $L_2$, $L_2^r$, $W(-1,1)^{-1}$, and $W(1,-1)$, then $\rho \in L_2$, there is at most one layer of size 2 in $\rho$, and that layer must either be the first layer or the last layer.  In other words, $\rho$ is trivially Wilf-equivalent to either $12\cdots k$ or $1\cdots (k-2)k(k-1)$, which is what we wanted to show.

Therefore, there are exactly 2 classes $\mathcal{R}_n(\rho)$ that have polynomial growth for $\rho \in \mathcal{S}_k$.
\end{proof}

\section{Summary}

In this paper, we have completely determined $\mathrm{r}_n(\rho)$ for any permutation pattern $\rho$ of length at most 4.  We have also determined the Wilf classes for patterns of length 5.  By realizing that $\mathrm{r}_n(\rho) = \mathrm{s}_n(\rho^{\leftrightarrow})$, we took advantage of earlier results in the permutation patterns literature to completely characterize when $\mathrm{r}_n(\rho)$ has polynomial growth.  We also modified a classic proof of the Erd\H{o}s--Szekeres Theorem to show that $\mathrm{r}_n(12\cdots k)=0$ for $n \geq \binom{k}{2}+1$, and used the Robinson--Schensted correspondence to determine $\mathrm{r}_{\binom{k}{2}}(12\cdots k)$ for all $k$.  There are still several open questions of interest:

\begin{enumerate}
\item All of the sequences in Table \ref{T:length4formula} have rational generating functions.  Do there exist patterns $\rho$ where the sequence $\{\mathrm{r}_n(\rho)\}$ does not have a rational generating function?
\item We know that the majority of the sequences in Table \ref{T:length5data} have rational generating functions because of Theorem \ref{T:finlabel}; however, actually computing the appropriate finitely-labeled generating trees was prohibitive both in terms of time and computer memory.  What other techniques can be used to enumerate the members of the corresponding permutation classes?
\item We determined $\mathrm{r}_{\binom{k}{2}}(12\cdots k)$ by characterizing the pairs of standard Young tableau that correspond to $\pi$ where $\pi\pi^r \in \mathcal{R}_{\binom{k}{2}}(12\cdots k)$.  While $\mathrm{r}_{\binom{k}{2}-1}(12\cdots k)=\mathrm{r}_{\binom{k}{2}}(12\cdots k)$ and $\mathrm{r}_{\binom{k}{2}-2}(12\cdots k)=\dfrac{\mathrm{r}_{\binom{k}{2}-1}(12\cdots k)}{2}$, for smaller length permutations $\pi$, $\pi\pi^r$ may avoid $12\cdots k$ without $P(\pi)$ having increasing diagonals.  What can be said about $P(\pi)$ where $\pi\pi^r \in \mathcal{R}_n(12\cdots k)$ for $n \leq k-3$?
\end{enumerate}

\acknowledgements
\label{S:ack}

The authors are grateful to two anonymous referees for their feedback, which improved the organization and clarity of this paper.

\nocite{*}
\bibliographystyle{abbrvnat}
\bibliography{rdlbib}
\label{sec:biblio}

\end{document}